\documentclass[reqno]{amsart}
\usepackage{enumitem}
\usepackage{mathrsfs, color}
\usepackage{amsfonts}
\usepackage{amssymb,amsmath}
\usepackage{amsthm}
\usepackage{graphics}
\usepackage{graphicx}
\usepackage{epsfig}
\usepackage[bookmarks=true]{hyperref}

\allowdisplaybreaks[1]

\newtheorem{theorem}{Theorem}[section]
\newtheorem{lemma}[theorem]{Lemma}

\newtheorem{condition}[theorem]{Condition}

\theoremstyle{definition}
\newtheorem{definition}[theorem]{Definition}

\newtheorem{proposition}[theorem]{Proposition}

\theoremstyle{remark}
\newtheorem{remark}[theorem]{Remark}

\numberwithin{equation}{section}

\newcommand{\abs}[1]{\lvert {#1} \rvert}
\newcommand{\norm}[1]{\lVert {#1} \rVert}
\newcommand{\pdt}[2]{\langle {#1},{#2} \rangle}

\newcommand{\mexp}[1]{\exp\Big\{{#1}\Big\}}
\newcommand{\bbra}[1]{\bigg({#1}\bigg)}
\newcommand{\set}[1]{\Big\{{#1}\Big\}}

\usepackage{algorithm} 
\usepackage{multirow} 
\usepackage{xcolor}
\usepackage{makecell}
\usepackage{lipsum}
\usepackage{amsfonts}
\usepackage{graphicx}
\usepackage{epstopdf}
\usepackage{algorithmic}
\ifpdf
  \DeclareGraphicsExtensions{.eps,.pdf,.png,.jpg}
\else
  \DeclareGraphicsExtensions{.eps}
\fi

\begin{document}
\setlength\abovedisplayskip{2pt}
\setlength\belowdisplayskip{1pt}
\setlength\abovedisplayshortskip{2pt}
\setlength\belowdisplayshortskip{1pt}

\title[Variational posterior convergence for inverse problems]
 {Consistency of variational inference for Besov priors in non-linear inverse problems}

\author[S.K.Zu]{Shaokang Zu}
\address{School of Mathematics and Statistics,
Xi'an Jiaotong University,
Xi'an,
710049, China}
\email{incredit1@stu.xjtu.edu.cn}

\author[J. Jia]{Junxiong Jia}
\thanks{Junxiong Jia is the corresponding author.}
\address{School of Mathematics and Statistics,
Xi'an Jiaotong University,
 Xi'an
710049, China}
\email{jjx323@xjtu.edu.cn}

\author[Z. Wang]{Zhiguo Wang}
\address{School of Mathematics and Statistics,
Xi'an Jiaotong University,
 Xi'an
710049, China}
\email{emailwzg@mail.xjtu.edu.cn}


\date{}


\begin{abstract}
This study investigates the variational posterior convergence rates of inverse problems for partial differential equations (PDEs) with parameters in Besov spaces \( B_{pp}^\alpha \) (\( p \geq 1 \)) which are modeled naturally in a Bayesian manner using Besov priors constructed via random wavelet expansions with \( p \)-exponentially distributed coefficients. Departing from exact Bayesian inference, variational inference transforms the inference problem into an optimization problem by introducing variational sets. Building on a refined ``prior mass and testing'' framework, we derive general conditions on PDE operators and guarantee that variational posteriors achieve convergence rates matching those of the true posterior under widely adopted variational families (Besov-type measures or mean-field families). Moreover, our results achieve minimax-optimal rates over $B^{\alpha}_{pp}$ classes, significantly outperforming the suboptimal rates of Gaussian priors (by a polynomial factor). As specific examples, two typical nonlinear inverse problems, the Darcy flow problems and the inverse potential problem for a subdiffusion equation, are investigated to validate our theory. Besides, we show that our convergence rates of ``prediction'' loss for these ``PDE-constrained regression problems'' are minimax optimal.
\end{abstract}

\maketitle



\section{Introduction}
The growing applications in many domains, such as seismic exploration and radar imaging, have driven substantial advances in inverse problems of partial differential equations (PDEs) over recent decades \cite{IntroPDE_haber2003learning}. Beyond deterministic solutions, the Bayesian approach has emerged as a powerful framework for uncertainty quantification through statistical inference \cite{IntroPDE_stuart2010inverse}. This approach reformulates inverse problems as statistical inference tasks, enabling rigorous characterization of parameter uncertainties in inverse problems of PDEs \cite{IntroPDE_Dashti2017}.
\par
To extract information from the posterior probability distribution, sampling methods such as Markov chain Monte Carlo (MCMC) methods are frequently utilized \cite{su_statistical_2023}. While MCMC is highly efficient and theoretically sound as a sampling method, its computational cost can become excessive for PDE-constrained likelihood evaluations \cite{Fichtner2011Book}. This bottleneck motivates variational Bayesian inference as a popular alternative. By minimizing the Kullback-Leibler (KL) divergence between a tractable family $\mathcal{Q}$ and the true posterior $\Pi(\cdot|D_N)$,  variational inference seeks to identify the variational posterior $\hat{Q}$, the closest approximation to $\Pi(\cdot\vert D_N)$. 
The structure of variational sets $\mathcal{Q}$ often allows variational inference methods to achieve comparable accuracy to MCMC with orders-of-magnitude speed improvements \cite{IntroVI_yang2020alpha}. Their optimization-based architecture particularly excels in large-scale inverse problems involving computationally intensive likelihood functions (see e.g. \cite{IntroVI_blei2017variational,IntroVI_meng2023sparse}). This approach is increasingly favored in Bayesian inverse problems, as illustrated by recent studies \cite{IntroVI_povala2022variational,IntroVI_jia2021variational} and their references. 
\par

While recent theoretical breakthroughs for Bayesian inverse problems by Nickl et al. \cite{IntroNonLinear_nickl2023bayesian} with Gaussian process priors and Agapiou et al. \cite{agapiou2024laplace} with Laplace priors establish posterior contraction rates, see also \cite{IntroNonLinear_nickl2020bernstein,IntroVI_meng2023sparse,zu2024consistencyvariationalbayesianinference},
current convergence analyses remain constrained to Sobolev spaces or $B_{11}^{\alpha}$, mismatched with general Besov spaces $B_{pp}^{\alpha}\ (p\geq 1)$.
Besov spaces provide a way to represent the unknown function in inverse problems with sharp edges, discontinuities, or varying degrees of smoothness, which are used in various inverse problems, including image reconstruction, geophysics, and density estimation. In Bayesian settings, many researchers have investigated inverse problems with Besov priors based on wavelets \cite{dashti2012besov,Jia_2016,lassas2009discretization}.
However, to our knowledge, no existing work establishes variational posterior $\hat{Q}$ convergence rates for nonlinear inverse problems in general Besov spaces, leaving a critical gap between methodological development and theoretical understanding.
\par
In this paper, we approach the variational posterior $\hat{Q}$ as the solution to the optimization problem \[ \mathop{\min}_{Q\in \mathcal{Q}} D(Q\Vert\Pi(\cdot|D_{N})),\] where $D(\cdot \Vert \cdot)$ represents the KL divergence, and $\Pi(\cdot|D_{N})$ denotes the posterior distribution derived from the data $D_N$ within a natural statistical observation model of the forward map $\mathcal{G}$ (refer to Section \ref{sec:GeneralSetting}). The primary goal of this study is to determine the convergence rate of the variational posterior $\hat{Q}$ towards the truth $\theta_0\in B_{pp}^{\alpha}$. We propose general conditions on the forward map $\mathcal{G}$, the prior and the variational class $\mathcal{Q}$ to describe this contraction. Assuming that the forward map $\mathcal{G}$ satisfies our regularity and conditional stability conditions, we demonstrate that, for Besov-type priors defined in Section \ref{sec:GeneralSetting} and the true parameter $\theta_0\in B_{pp}^{\alpha}$, the variational posterior $\hat{Q}$ converges to the true parameter at the rate specified by \[\varepsilon_N^2 + \frac{1}{N}\mathop{\inf}_{Q \in \mathcal{Q}}P_{0}^{(N)}D(Q\Vert \Pi_N(\cdot|D_N)),\]
\textcolor{black}{which originates from the general contraction theory of variational posterior \cite{zhang2020convergence}.}
The first term, $\varepsilon_N^2$, represents the convergence rate of the posterior $\Pi_N(\cdot|D_N)$ and is ordered as $N^{-a}$ for some $a > 0$. The second term represents the variational approximation error arising from the data-generating process $P_{0}^{(N)}$, which is induced by the true parameter. When the variational set $\mathcal{Q}$ comprises Besov-type measures or a mean-field variational class, we demonstrate that the variational approximation error is dominated by $\varepsilon_N^2$ (up to a logarithmic factor). Therefore, this implies that the convergence rates of the variational posterior distributions for nonlinear inverse problems can attain the same convergence rates as the posterior distributions (up to a logarithmic factor). Besides, in a setting with direct observations ($\mathcal{G}\equiv $Id), our convergence rate achieves the minimax rate  (up to a logarithmic factor) over Besov function classes $B_{pp}^{\alpha}$ for $p\in [1,2)$, while Gaussian process priors can only achieve polynomially slower convergence rates (see Section 4 of \cite{agapiou2024laplace} for details). In Section 4, we derive the convergence rates of the variational posterior for the Darcy flow problem and the inverse potential problem for a subdiffusion equation. We also show that the convergence rates of ``prediction'' loss for these ``PDE-constrained regression problems''  are minimax optimal.
\subsection*{Related work}
The theory behind Bayesian methods for linear inverse problems is now well-founded. Initial research into asymptotic behavior focused on conjugate priors \cite{IntroLinear_knapik2011bayesian}, and later extended to non-conjugate priors \cite{IntroLinear_ray2013bayesian}. See \cite{IntroLinear_knapik2013bayesian,IntroLinear_agapiou2014bayesian,IntroLinear_agapiou2013posterior,IntroLinear_jia2021posterior} for more references. In recent years, the theory of Bayesian nonlinear inverse problems has seen significant advancements. 
Nickl et al. \cite{IntroNonLinear_nickl2020convergence} present the theory for the convergence rate of maximum a posterior estimates with Gaussian process priors, providing examples involving the Darcy flow problem and the Schrödinger equation. For
X-ray transforms, Monard et al. \cite{IntroNonLinear_monard2019efficient,IntroNonLinear_monard2021consistent} prove Bernstein–von Mises theorems for a large family of one-dimensional linear functionals of the target parameter, and show the convergence rate of the statistical error in the recovery of a matrix field with Gaussian process priors. Subsequently, Giordano and Nickl \cite{IntroNonLinear_giordano2020consistency} demonstrated the convergence rate of the posterior with Gaussian process priors for an elliptic inverse problem. In 2020, Abraham and Nickl \cite{IntroNonLinear_abraham2020statistical} showed that a statistical algorithm constructed in the paper recovers the target parameter in supremum-norm loss at a statistical optimal convergence rate of the logarithmic order. For the Schr\"{o}dinger equation, Nickl \cite{IntroNonLinear_nickl2020bernstein} established the Bernstein–von Mises theorem for the posterior distribution using a prior based on a specific basis. For additional references, see \cite{IntroNonLinear_nickl2023some,IntroNonLinear_bohr2022bernstein,IntroNonLinear_monard2021statistical} and the monograph \cite{IntroNonLinear_nickl2023bayesian}.
\par
Variational Bayes inference has seen extensive use across various fields. Recently, theoretical results start to surface. In \cite{IntroVI_wang2019frequentist}, Wang and Blei established Bernstein–von Mises type of results for parametric models. For other related references on theories on parametric variational Bayes inference, refer to \cite{IntroVI_blei2017variational}. Regarding nonparametric and high-dimensional
models that are more relevant to non-linear inverse problems, recent researches \cite{IntroVI_yang2020alpha,IntroVI_alquier2020concentration} investigated variational approximation to tempered posteriors.
In the context of variational Bayesian inference for usual posterior distributions, studies \cite{zhang2020convergence,IntroVI_pati2018statistical} achieved results comparable to those obtained for tempered posteriors. When it comes to inverse problems, some algorithmic developments can be found in \cite{IntroVI_jia2021variational,IntroVI_jin2012variational,IntroVI_guha2015variational}. For theoretical results of linear inverse problems, Randrianarisoa and Szab\'o \cite{IntroLinear_randrianarisoa2023variational} employed the inducing variable method to derive the contraction rates of variational posterior with Gaussian process priors around the true function. Concerning non-linear inverse problems, study \cite{zu2024consistencyvariationalbayesianinference} extends these theoretical results of variational posterior to non-linear inverse problem settings with Gaussian priors.

{\color{black} Our present work extends the prior framework in \cite{zu2024consistencyvariationalbayesianinference} from Gaussian priors to the broader class of Besov priors generated by $p$-exponential wavelet coefficients. From our previous work, we adopt the same statistical model with Gaussian noise, the general conditions on the forward map $\mathcal{G}$ (local Lipschitz continuity and conditional stability, now formulated in Besov norms), and the overall proof strategy, which combines variational contraction theory from \cite{zhang2020convergence} with posterior concentration techniques for nonlinear inverse problems from \cite{IntroNonLinear_nickl2023bayesian}.

The new contributions relative to \cite{zu2024consistencyvariationalbayesianinference} are as follows. We establish new concentration results for Besov priors. For variational inference, we introduce a novel $p$-exponential mean‑field family $\mathcal{Q}_E$ tailored to the Besov prior, and prove that its approximation error is controlled by the contraction rate of the true posterior. We also apply the theory to the Darcy flow and subdiffusion inverse problems, obtaining minimax optimal contraction rates in Besov spaces $B_{pp}^\alpha$ for the ``PDE‑constrained regression problem'', whereas the rates are polynomially slower under Gaussian priors.}

\subsection*{Organization}
The structure of this paper is as follows. In Section \ref{sec:GeneralSetting}, we begin by introducing the notations and statistical settings for the inverse regression model. We also define Besov-type priors on $L^2$. In Section \ref{sec:VariationalConsistency}, we establish a theory that guarantees convergence rates of variational Bayesian inference with Besov-type priors for nonlinear inverse problems. In Section \ref{sec:ApplicationSection}, we demonstrate the application of our theory to the Darcy flow problem and the inverse potential problem for a subdiffusion equation, showing that the convergence rates of ``prediction'' loss for these ``PDE-constrained regression problems'' are minimax optimal (up to a logarithmic factor). Proofs for results of inverse problems and properties of Besov priors are given in Sections \ref{sec:proofIP} and \ref{sec:proofBesov} respectively.

\subsection*{Basic notations}
For $\mathcal{X} \subset \mathbb{R}^d$ an open set, $L^2(\mathcal{X})$ denotes the standard Hilbert space of square integrable functions on $\mathcal{X}$ with respect to Lebesgue measure, with inner product 
$\pdt{\cdot}{\cdot}_{L^2}$ and $\norm{\ \cdot\ }_{L^2}$ norm where we omit $\mathcal{X}$ when no confusion may arise. $L^2_{\lambda}(\mathcal{X})$ denotes the
space of $\lambda$-square integrable function on $\mathcal{X}$, where $\lambda$ is a measure on $\mathcal{X}$.
\par
For a multi-index $i = (i_1, ..., i_d)$, $i_j = \mathbb{N}\cup\{0\}$. Let $D^i$ denote the $i$-th (weak) partial differential operator of order $\abs{i} = \sum_j i_j$. Then the Sobolev spaces are defined as
\begin{align*}
H^{\alpha}(\mathcal{X}) = \set{{f \in L^2(\mathcal{X}) : D^i f \in L^2(\mathcal{X}) \ \forall |i| \leq \alpha}},\quad \alpha \in \mathbb{N}
\end{align*}
normed by $\norm{f}_{H^{\alpha}(\mathcal{X})} = \sum_{\abs{i}\leq \alpha} \norm{D^i f}_{L^2(\mathcal{X})}.$
For a non-integer real number $\alpha >0$, one defines $H^{\alpha}$ by interpolation.
For $\alpha \geq 0$ and $p,q \in [1,\infty]$, we denote by $B^{\alpha}_{pq}(\mathcal{X})$ the usual $\alpha$-regular space of Besov functions on $\mathcal{X}$ (see \cite{triebel1996function} for detailed definitions). 
\par
The space of bounded and continuous functions on $\mathcal{X}$ is denoted by $C^0(\mathcal{X})$,
equipped with the supremum norm $\norm{\ \cdot\ }_{\infty}$. For $\alpha \in \mathbb{N}\cup\{0\}$, the space of $\alpha$-times differentiable functions on $\mathcal{X}$ can be similarly defined as
\begin{align*}
C^{\alpha}(\mathcal{X}) = \set{{f \in C^0(\mathcal{X}) : D^i f \in C^0(\mathcal{X}) \ \forall |i| \leq \alpha}},\quad \alpha \in \mathbb{N}
\end{align*}
normed by
$\norm{f}_{C^{\alpha}(\mathcal{X})} = \sum_{\abs{i}\leq \alpha} \norm{D^i f}_{\infty}.$ 
The symbol $C^{\infty}(\mathcal{X})$ denotes the set of all infinitely differentiable functions on $\mathcal{X}$.
For a non-integer real number $\alpha >0$, we say $f \in C^{\alpha}(\mathcal{X})$ if for all multi-indices $i$ with $\abs{i}\leq \lfloor \alpha \rfloor$, $D^if$ exists and is $\alpha - \lfloor \alpha \rfloor$-H$\ddot{\mathrm{o}}$lder continuous. The norm on the space $C^{\alpha}(\mathcal{X})$ for such $\alpha$ is given by
\begin{align*}
\norm{f}_{C^{\alpha}(\mathcal{X})} = \norm{f}_{C^{\lfloor \alpha \rfloor}(\mathcal{X})} + \sum_{\abs{i} = \lfloor \alpha \rfloor} \sup_{x,y\in\mathcal{X}, x\neq y}\frac{\abs{D^i f(x)-D^i f(y)}}{\abs{x-y}^{\alpha - \lfloor \alpha \rfloor}}. 
\end{align*}
\par
For any space $S(\mathcal{X})$ with norm $\norm{\ \cdot \ }_{S(\mathcal{X})}$, we will sometimes omit $\mathcal{X}$ in the
notation and denote the norm by $\norm{\ \cdot \ }_{S}$ when no confusion may arise. The notation $S_0(\mathcal{X})$ denotes the subspace $(S_0(\mathcal{X}),\norm{\ \cdot \ }_{S(\mathcal{X})})$, consisting of elements of $S(\mathcal{X})$ that vanish at $\partial\mathcal{X}$. The notation $S_c(\mathcal{X})$ denotes the subspace $(S_c(\mathcal{X}),\norm{\ \cdot \ }_{S(\mathcal{X})})$, consisting of elements of $S(\mathcal{X})$ that are compactly supported in $\mathcal{X}$. Similarly, $S_K(\mathcal{X})$ denotes the subspace $(S_K(\mathcal{X}),\norm{\ \cdot \ }_{S(\mathcal{X})})$, consisting of elements of $S(\mathcal{X})$ that are supported in a subset $K\subset \mathcal{X}$. We use $S(\mathcal{X})^*$ to denote the dual space of $S(\mathcal{X})$. $B_{S(\mathcal{X})}(r)$ denotes $\{x \in S(\mathcal{X}) : \norm{x}_{S(\mathcal{X})} < r\}$ for $r\geq0$.
\par
All preceding spaces and norms can be defined for vector fields $f : \mathcal{X} \rightarrow \mathbb{R}^{d'}$ with
standard modification of the norms, by requiring each of the coordinate functions $f_i(\cdot), i = 1,\dots,d'$, to belong to the corresponding space of real-valued maps. We denote preceding spaces $S(\mathcal{X})$ defined for vector fields by $S(\mathcal{X},\mathbb{R}^{d'})$.
\par
Throughout the paper, $C, c$ and their variants denote generic constants that are either universal or ``fixed'' depending on the context.
For $a, b \in\mathbb{R}$, let $a\vee b = \max(a,b)$ and $a\wedge b = \min(a,b)$. The relation $a\lesssim b$ denotes an inequality $a\leq Cb$, and the corresponding convention is used for $\gtrsim$. The relation $a \simeq b$ holds if both $a\lesssim b$ and $a\gtrsim b$ hold. We define $\mathbb{R}^+=\{x\in\mathbb{R}: x\geq 0\}$. For $x\in \mathbb{R}^+$, $\lceil x \rceil$ is the smallest integer no smaller than $x$ and $\lfloor x \rfloor$ is the largest integer no larger than $x$. Given a set $S$, $\abs{S}$ denotes its cardinality, and $\textbf{1}_{S}$ is  the associated indicator function. The notations $\mathbb{P}$ and $\mathbb{E}$ are used to denote generic probability and expectation respectively, whose distribution is determined from the context. Additionally, the notation $\mathbb{P}f$ also means expectation of $f$ under $\mathbb{P}$, that is $\mathbb{P}f = \int f d\mathbb{P}$.
\par
For any $\alpha > d/p$, $0< \beta < \alpha - d/p$ and $\mathcal{X}$ a bounded domain with smooth boundary, the Sobolev imbedding implies that $B^{\alpha}_{pp}(\mathcal{X})$ embeds continuously into $C^{\beta}(\mathcal{X})$ (see Section 4.3.4 in \cite{Triebel2010theory}), with norm estimates
\begin{align}
\norm{f}_{\infty} \lesssim \norm{f}_{C^{\beta}} \lesssim \norm{f}_{B^{\alpha}_{pp}}, \quad \forall f\in B^{\alpha}_{pp}(\mathcal{X}).
\end{align}
We repeatedly use the inequalities from \cite[Remark 1 on p.143 and Theorem 2.8.3]{Triebel2010theory}
\begin{align}
\norm{fg}_{B^{\alpha}_{pp}} & \lesssim \norm{f}_{B^{\alpha}_{pp}}\norm{g}_{B^{\alpha}_{pp}}, \qquad \alpha >d/p, \label{Sobolevinter1} \\
\norm{fg}_{B^{\alpha}_{pp}} & \lesssim \norm{f}_{C^{\alpha}}\norm{g}_{B^{\alpha}_{pp}}, \qquad \alpha \geq 0. \label{Sobolevinter2}
\end{align}

\section{Statistical setting for general inverse problems}\label{sec:GeneralSetting}
\subsection{Forward map and variational posterior}    
    Let $(\mathcal{X},\mathcal{A})$ and $(\mathcal{Z},\mathcal{B})$ be measurable spaces equipped with probability measure $\lambda$ and Lebesgue measure respectively. Let $\mathcal{Z}$ be a bounded smooth domain in $\mathbb{R}^d$. Moreover, let $V$ be a vector space of fixed finite dimension $p_V \in \mathbb{N}$, with inner product $\pdt{\cdot}{\cdot}_V$ and norm $\norm{\ \cdot\ }_V$. For parameter $f$, we assume that the parameter spaces $\mathcal{F}$ is a Borel-measurable subspace of $L^2(\mathcal{Z},\mathbb{R})$. We define $G:\mathcal{F} \rightarrow L^2_{\lambda}(\mathcal{X},V)$ to be a measurable forward map. However, the parameter spaces $\mathcal{F}$ may not be linear spaces. In order to use Besov priors that are naturally supported in linear spaces such as Sobolev spaces, we now consider a bijective reparametrization of $\mathcal{F}$ through a regular link function $\Phi$ as in \cite{IntroNonLinear_nickl2020convergence,IntroNonLinear_giordano2020consistency}. The parameter spaces can be reparametrized as $\mathcal{F} = \{f = \Phi(\theta)\ : \ \theta \in \Theta\}$ where $\Theta$ is a subspace of $L^2(\mathcal{Z},\mathbb{R})$. For the forward map $G:\mathcal{F} \rightarrow L^2_{\lambda}(\mathcal{X},V)$, we define the reparametrized forward map $\mathcal{G}$ by 
\[\mathcal{G}(\theta) = G(\Phi(\theta)), \qquad \forall \ \theta \in \Theta,\] 
and consider independent and identically distributed (i.i.d.) random variables $(Y_i,X_i)_{i=1}^N$ of the following random design regression model
\begin{equation}\label{model}
    Y_i = \mathcal{G}({\theta})(X_i) + \varepsilon_i, \quad \varepsilon_i \mathop{\sim}^{iid} N(0,I_V), \quad i = 1, 2, \dots, N,
\end{equation}
with the diagonal covariance matrix $I_V$.
The joint law of the random variables $(Y_i,X_i)_{i=1}^N$ defines a product probability measure on $(V\times\mathcal{X})^N$ and we denote it by 
$P_{\theta}^{(N)}=\bigotimes_{i=1}^NP_{\theta}^i$ where $P_{\theta}^i=P_{\theta}^1$ for all $i$. We write $P_{\theta}$ for the law of a copy $(Y,X)$ which has probability density 
\begin{equation}\label{modeldensity}
    \frac{dP_{\theta}}{d\mu}(y,x) \equiv p_{\theta}(y,x) = \frac{1}{(2\pi)^{p_V/2}}\mexp{-\frac{1}{2}\abs{y-\mathcal{G}_{\theta}(x)}_V^2}
\end{equation}
for dominating measure $d\mu = dy \times d\lambda$ where $dy$ is Lebesgue measure on $V$. We also use the notation
\[D_N := \set{(Y_i,X_i) : i = 1, 2, \dots, N}, \qquad N\in \mathbb{N}\]
to represent the full data.
\par
Now we introduce the variational posterior distribution. Let prior $\Pi$ be a Borel probability measure on $\Theta$. The posterior distribution  $\Pi_N(\cdot|D_N) = \Pi_N(\cdot|(Y_i,X_i)_{i=1}^N)$ of $\theta|(Y_i,X_i)_{i=1}^N$ on $\Theta$ arising from the data in model (\ref{model}) is given by
\begin{equation}\label{Post}
    d\Pi(\theta|D_N) = \frac{\prod^{N}_{i=1}p_{\theta}(Y_i,X_i)d\Pi(\theta)}{\int_{\Theta}\prod^{N}_{i=1}p_{\theta}(Y_i,X_i)d\Pi(\theta)}.
\end{equation}
To address computational difficulty of forward problems in posterior distributions, variational inference aim to find the closest element to the posterior distribution in a variational set $\mathcal{Q}$ consisting of probability measures. The most popular definition of variational inference is given through the KL divergence defined as
\begin{equation*}
    D(P_1\Vert P_2) = \left\{ \begin{aligned}
                &\int \log\bbra{\frac{dP_1}{dP_2}} dP_1 & \text{if}\ P_1 \ll P_2,\\
                 &+ \infty & \text{otherwise}.
            \end{aligned} \right.
\end{equation*} Then, the variational posterior is defined as 
\begin{equation}\label{variationalposterior}
    \hat{Q} = \mathop{\mathrm{argmin}}_{Q\in \mathcal{Q}} D(Q\Vert \Pi(\theta|D_{N})).
\end{equation}
\par
The choice of variational set $\mathcal{Q}$ usually determines the effect of variational posterior. The variational posterior from the variational set $\mathcal{Q}$ can be regarded as the projection of the true posterior onto $\mathcal{Q}$ under KL-divergence. When $\mathcal{Q}$ is large enough (even consisting of all probability measures), $\hat{Q}$ is exactly the true posterior $\Pi(\theta|D^{N})$. However, a larger $\mathcal{Q}$ often leads to a higher computational cost of the optimization problem (\ref{variationalposterior}). Thus, it is of significance to choose a proper variational set $\mathcal{Q}$ in order to reduce the computational difficulty and give a good approximation of posterior simultaneously. We will give the convergence rate of variational posteriors with a mean-field variational set, see Section \ref{sec:VariationalConsistency}.
\subsection{Besov-type priors}
In this paper we consider Besov-type priors on $L^2(\mathcal{Z})$ which arise as random wavelet series expansions with weights distributed as p-exponential distributions, see also \cite{agapiou2024laplace,IntroNonLinear_nickl2020convergence}. First, we introduce an orthonormal wavelet basis of the Hilbert space $L^2(\mathbb{R}^d)$ as
\begin{equation} \label{Dbase}
\left\{\psi_{lr}: r \in \mathbb{Z}^d, l \in \mathbb{N}\cup\{-1,0\}\right\}
\end{equation}
composed of compactly supported Daubechies wavelets with sufficiently large regularity $S>0$ (for details see Chapter 4 in \cite{gin2015mathematical}). 
For a compact set $K\subset\mathcal{Z}\subset\mathbb{R}^d$, We further denote by
\begin{equation} \label{DbaseK}
\Psi_K = \left\{\psi_{lr}: r \in R_l, l \in \mathbb{N}\cup\{-1,0\}\right\}
\end{equation}
the sub-basis of wavelets, where $R_l$ denotes the set of indices $r$ for which the support of $\psi_{lr}$ intersects $K$.
We note that $R_l$ satisfies $\abs{R_l}\simeq 2^{ld}$. Any function $f\in L^2_K(\mathcal{Z})$ can be uniquely represented by this sub-basis as
\[f= \sum_{l=-1}^{+\infty}\sum_{r\in R_l}f_{lr}\psi_{lr}, \qquad f_{lr}=\pdt{f}{\psi_{lr}}_{L^2(\mathcal{Z})}.\]
Then, we give definitions of univariate p-exponential distributions and Besov priors.
\begin{definition}[Univariate p-exponential distributions]
    Let $p\geq1, a\in \mathbb{R}, b>0$. A real random variable $\xi$ is called a univariate p-exponential variable with location $a$ and scale $b$ if it has a probability density function $f(x)\propto \exp(-\frac{\abs{x-a}^p}{pb^p})$, which is represented by $\xi \sim Exp(p;a,b)$. Particularly, we call $\xi$ a univariate standard p-exponential variable if $\xi\sim Exp(p;0,1)$.
\end{definition}
\begin{definition}[$B_{pp}^{\alpha}$-Besov priors]\label{Def:BesovPrior}
     For $p\geq 1$, regularity level $\alpha>\frac{d}{p}$, regularity of wavelets $S>\alpha$ and i.i.d univariate standard p-exponential random variables $\{\xi_{lr}\}_{l\geq -1, r\in R_l},$ let
     \begin{equation}\label{Pi'}
         F = \sum_{l=-1}^{+\infty}\sum_{r\in R_l}2^{l(\frac{d}{p}-\frac{d}{2}-\alpha)}\xi_{lr}\psi_{lr}.
     \end{equation}
    Denote the law of $F$ by $\Pi'$. For the given compact set $K$ corresponding to $\Psi_K$, fix a cut-off function $\chi \in C_c^{\infty}(\mathcal{Z})$ such that $\chi = 1$ on $K$. Then, for some scaling constant $\rho$, the $B_{pp}^{\alpha}$-Besov priors $\Pi$ is defined by
    \begin{equation}\label{Pi}
        \Pi = \mathcal{L}(\rho\chi F), \quad F\sim\Pi'.
    \end{equation}
\end{definition}

We see that when we set $p=1$ and $p=2$, the general Besov prior reduced to Laplace and Gaussian wavelet priors respectively.  Convergence rates of true posterior for inverse problems with these specific priors have been considered in \cite{agapiou2024laplace,IntroNonLinear_nickl2020convergence}.
Regarding variational posteriors, \cite{zu2024consistencyvariationalbayesianinference} gives the convergence rates for non-linear inverse problems with Gaussian priors. 
For the more general situation $p\geq1$ in our paper, we prove that analogous properties used in the proof also hold for general Besov priors, which leads to similar convergence results. However, the decentering property (see Lemma \ref{lem:decenter}) is restricted to $p\in [1,2]$. Thus, when $p>2$, it is necessary for us to consider the truncated Besov priors which satisfy an analogous decentering property (see Lemma \ref{lem:decenterJ} for details).
\begin{definition}[Truncated $B_{pp}^{\alpha}$-Besov priors]\label{Def:TrunBesovPrior}
     For $p\geq 1$, regularity level $\alpha>\frac{d}{p}$, regularity of wavelets $S>\alpha$, truncated point $J\in \mathbb{N}$ and i.i.d univariate standard p-exponential random variables $\{\xi_{lr}\}_{l\geq -1, r\in R_l},$ let
     \begin{equation}\label{Pi'J}
         F_J = \sum_{l=-1}^{J}\sum_{r\in R_l}2^{l(\frac{d}{p}-\frac{d}{2}-\alpha)}\xi_{lr}\psi_{lr}.
     \end{equation}
    Denote the law of $F_J$ by $\Pi'_J$. For the given compact set $K$ corresponding to $\Psi_K$, fix a cut-off function $\chi \in C_c^{\infty}(\mathcal{Z})$ such that $\chi = 1$ on $K$. Then, for some scaling constant $\rho$, the truncated $B_{pp}^{\alpha}$-Besov priors $\Pi$ is defined by
    \begin{equation}\label{PiJ}
        \Pi = \mathcal{L}(\rho\chi F_J), \quad F_J\sim \Pi'_J.
    \end{equation}
\end{definition}
    \section{Variational posterior consistency theorem} \label{sec:VariationalConsistency}
    In this section, assuming the data $D_N$ to be generated through model (\ref{model}) of law $P^{(N)}_{\theta_0}$ and the forward map $\mathcal{G}$ to satisfy certain conditions, we will show that the variational posterior distribution arising from certain Besov priors concentrates near any sufficiently regular ground truth $\theta_0$ (or, equivalently, $f_0 = \Phi(\theta_0)$), and show the rate of this contraction.
\subsection{Conditions for the forward map}
    In this part, we give the regularity and conditional stability conditions which resemble the conditions in \cite{IntroNonLinear_nickl2023bayesian}. The following regularity condition requires the forward map $\mathcal{G}$ to be uniformly bounded and Lipschitz continuous. Compared to the conditions in \cite{IntroNonLinear_nickl2023bayesian},  we further require polynomial growth in parameter's norm of uniform upper bound and the Lipschitz constants.
\begin{condition}[Locally Lipschitz]\label{con:condreg}
     Consider a parameter space $\Theta \subseteq L^2(\mathcal{Z},\mathbb{R})$. 
  The forward map $\mathcal{G}:\Theta \rightarrow L^2_{\lambda}(\mathcal{X},V)$ is measurable. 
  Suppose for some normed linear subspace $(\mathcal{R},\Vert \cdot\Vert _{\mathcal{R}})$ of $\Theta$ and all $M>1$, there exist finite constants $C_U > 0$, $C_L > 0$ , ${\kappa}\geq0$, $\mu\geq 0$ and $l \geq 0$ such that 
  \begin{gather}
          \mathop{\sup}_{\theta \in \Theta \cap B_{\mathcal{R}}(M)}\mathop{\sup}_{x \in \mathcal{X}}\abs{\mathcal{G}(\theta)(x)}_V\leq C_{U}M^\mu,\label{bound}\\
          \norm{\mathcal{G}(\theta_1)-\mathcal{G}(\theta_2)}_{L^{2}_{\lambda}(\mathcal{X},V)} \leq C_L(1+\norm{\theta_1}^l_{\mathcal{R}}\vee \norm{\theta_2}^l_{\mathcal{R}})\norm{\theta_1-\theta_2} _{(H^{{\kappa}})^*}, \quad \theta_1,\theta_2 \in \mathcal{R}.\label{lip}
      \end{gather}
\end{condition}
Similarly, the conditional stability condition is given below.
\begin{condition}[Stability]\label{con:condstab}
    Consider a parameter space $\Theta \subseteq L^2(\mathcal{Z},\mathbb{R})$. 
  The forward map $\mathcal{G}:\Theta \rightarrow L^2_{\lambda}(\mathcal{X},V)$ is measurable. 
  Suppose for some normed linear subspace $(\mathcal{R},\Vert \cdot\Vert _{\mathcal{R}})$ of $\Theta$, some ground truth $\theta_0$ and all $M>1$, there exist a function $F: \mathbb{R}^{+} \rightarrow \mathbb{R}^{+}$ and finite constants $C_T > 0$, $\nu\geq 0$ such that
        \begin{equation}\label{stab}
              F(\Vert f_{\theta} - f_{\theta_0}\Vert ) \leq C_TM^\nu\Vert \mathcal{G}(\theta)-\mathcal{G}(\theta_0)\Vert _{L^{2}_{\lambda}(\mathcal{X},V)}, \forall \theta \in \Theta \cap B_{\mathcal{R}}(M).
          \end{equation}
\end{condition}
The linear subspace $(\mathcal{R},\Vert \cdot\Vert _{\mathcal{R}})$ in Conditions \ref{con:condreg} and \ref{con:condstab} is usually chosen as the support of Besov priors, such as $(B_{pp}^b(\mathcal{Z}),{\Vert \cdot\Vert _{B_{pp}^b(\mathcal{Z})}})$ for $b<\alpha-\frac{d}{p}$.
These regularity and stability conditions on the forward map $\mathcal{G}$ require explicit estimates on growth rates of coefficients. We will verify these conditions for the inverse problems related to the Darcy flow equation and a subdiffusion equation in Section \ref{sec:ApplicationSection}.
\subsection{Results for rescaled Besov priors}
    Consider the centred Besov probability measure $\Pi'$ defined by \eqref{Pi'}. We built the recaled Besov priors $\Pi_N$ through a $N$-dependent rescaling step to $\Pi'$ as in \cite{IntroNonLinear_giordano2020consistency}.
        Let $\Pi'$ be the centred Besov probability measure $\Pi'$ defined in Definition \ref{Def:BesovPrior}. For $\kappa$ in Condition \ref{con:condreg}, let $\Pi_N$ be the corresponding $B_{pp}^{\alpha}$-Besov prior from Definition \ref{Def:BesovPrior}, with scaling constant
        \[\rho = (N\varepsilon_N^2)^{-\frac{1}{p}}, \qquad \varepsilon_N = N^{-\frac{\alpha+\kappa}{2\alpha+2\kappa+d}},\]
        where $N$ is the number of data points from the model \eqref{model}.
    We see the prior $\Pi = \Pi_N$ arises from the base prior $\Pi'$ defined as the law of
    \begin{equation}\label{rescaledprior}
        \theta = N^{-\frac{d}{p(2\alpha+2\kappa+d)}}\theta',
    \end{equation}
    for $\theta' \sim \Pi'$, ``regularity'' parameter $\alpha \geq 0$ of the ground truth $\theta_0$ and $\kappa$ the ``forward smoothing degree'' of $\mathcal{G}$ in condition \ref{con:condreg}.
    The first result shows that the variational posterior converge towards the ground truth $\theta_0$ with an explicit rate.
    \begin{theorem} \label{mainthm}
      Consider $\theta_0\in B_{pp}^{\alpha}(\mathcal{Z})$ supported in the compact set $K$ for some $\alpha>\frac{d}{p}$ and $p\in[1,2]$.
      Suppose Condition \ref{con:condreg} holds for forward map $\mathcal{G}$,  separable normed linear subspace $(\mathcal{R},\Vert \cdot\Vert _{\mathcal{R}}) = (B_{pp}^b(\mathcal{Z}),\Vert \cdot\Vert _{B_{pp}^b(\mathcal{Z})})$ with some $b<\alpha-\frac{d}{p}$ and finite constants $C_U > 0$, $C_L > 0$ , $\kappa\geq0$, $\mu\geq 0$ and $l \geq 0$.
      Denote by $\Pi_N$ the rescaled prior as in (\ref{rescaledprior}) with $\theta' \sim \Pi'$, and $\Pi_N(\cdot|(Y_i,X_i)_{i=1}^N) = \Pi_N(\cdot|D_N)$ is the corresponding posterior distribution in (\ref{Post}) arising from data in the model (\ref{model}).
      Assume that $\alpha + \kappa \geq \frac{d(l+1)}{p}$. Then, for $\varepsilon_N=N^{-\frac{\alpha+\kappa}{2\alpha+2\kappa+d}}$, $\gamma_N^2 = \frac{1}{N}\mathop{\inf}_{Q \in \mathcal{S}}P_{\theta_0}^{(N)}D(Q\Vert \Pi_N(\cdot|D_N))$ and variational posterior $\hat{Q}$ defined in (\ref{variationalposterior}), we have
      \begin{equation}
          P_{\theta_0}^{(N)}\hat{Q}\Vert \mathcal{G}(\theta)-\mathcal{G}(\theta_0)\Vert _{L^{2}_{\lambda}}^{\frac{2}{\mu+1}}\lesssim \varepsilon_N^{\frac{2}{\mu+1}} +\gamma^2_N \cdot \varepsilon_N^{-\frac{2\mu}{\mu+1}}.
      \end{equation}
      Moreover, assume that Condition \ref{con:condstab} also holds for $\mathcal{G}$, $\mathcal{R}$, function $F$ and finite constants $C_T > 0$, $\nu\geq 0$. Then, we further have
      \begin{equation}
          P_{\theta_0}^{(N)}\hat{Q}[F(\Vert f_{\theta} - f_{\theta_0}\Vert )]^{\frac{2}{\mu+\nu+1}}\lesssim \varepsilon_N^{\frac{2}{\mu+\nu+1}} +\gamma^2_N \cdot \varepsilon_N^{-\frac{2\mu+2\nu}{\mu+\nu+1}}.
      \end{equation}
    \end{theorem}
        The convergence rate can be represented as $(\varepsilon_N^2+\gamma_N^2) \cdot \varepsilon_N^{-\eta}$ for some nonnegative $\eta < 2$. We can see that the convergence rate is controlled by two terms $\varepsilon_N^2$ and $\gamma_N^2$. The first term $\varepsilon_N^2$ is the convergence rate of the true posterior $\Pi(\theta|D_N)$ and it can reach the same rate in \cite{IntroNonLinear_nickl2023bayesian}. The second term $\gamma_N^2$ characterizes the approximation error given by the variational set $\mathcal{Q}$. A larger $\mathcal{Q}$ leads to a smaller rate $\gamma_N^2$.
        To use Theorem \ref{mainthm} in specific problems, we need to bound the variational approximation error $\gamma_N^2$. In \cite{zhang2020convergence}, the authors give an upper bound 
        \[\gamma_N^2 \leq \mathop{\inf}_{Q \in \mathcal{Q}}R(Q)\]
        where \[R(Q) = \frac{1}{N}(D(Q\Vert \Pi) + Q[D(P^{(N)}_{\theta_0}\Vert P^{(N)}_{\theta})]).\]
        We estimate $\mathop{\inf}_{Q \in \mathcal{Q}}R(Q)$ in the next theorem.
        Define a p-exponential mean-field family \textcolor{black}{$\tilde{\mathcal{Q}}_E$} as
\begin{equation}\label{OGMFinf}
\Bigg\{\textcolor{black}{\tilde{Q}} = \mathop{\bigotimes}_{l=-1}^{+\infty}\mathop{\bigotimes}_{r\in R_l}Exp(p;a_{lr},b_{lr}):a_{lr} \in \mathbb{R}, b_{lr} \geq 0\Bigg\}.
\end{equation}
We define a map $\Psi$ as
\begin{equation}\label{pushforwardmapinf}
\Psi(\tilde{\theta}) = \sum_{l=-1}^{\infty}\sum_{r\in R_l}\tilde{\theta}_{lr}\chi\psi_{lr},\quad \forall \tilde{\theta} = (\tilde{\theta}_{lr})\in \mathbb{R}^{\infty},
\end{equation}
which maps any sequences $\tilde{\theta}\in \mathbb{R}^{\infty}$ to the function space span$\{\chi\psi_{lr}\}$.
Thus, our variational set $\mathcal{Q}_E$ is obtained by the push-forward of $\tilde{\mathcal{Q}}_E$ via $\Psi$, i.e.,
\begin{equation}\label{GMFinf}
\mathcal{Q}_E = \Bigg\{Q = \tilde{Q}\circ\Psi^{-1} : \tilde{Q} \in \tilde{\mathcal{Q}}_E\Bigg\}.
\end{equation}
We can find a probability measure $Q\in \mathcal{Q}_E$ such that $R(Q) \lesssim \varepsilon_N^2\log N$, which gives the convergence rate in the following theorem.
        \begin{theorem} \label{boundgam}
            Suppose that the forward map $\mathcal{G}$ satisfies Condition \ref{con:condreg} and Besov prior $\Pi_N$, the true parameter $\theta_0$ are as defined in Theorem \ref{mainthm}, then there exists a probability measure $Q_N\in \mathcal{Q}_E$ such that
            \begin{equation}
                R(Q_N) \lesssim \varepsilon_N^2\log N
            \end{equation}
            for $\varepsilon_N=N^{-\frac{\alpha+\kappa}{2\alpha+2\kappa+d}}$.
        \end{theorem}
        \par
        From Theorem \ref{boundgam}, if the variational set $\mathcal{Q}$ contains the probability measure $Q_N$, we can give a bound 
        \[\gamma_N^2 \leq \mathop{\inf}_{Q \in \mathcal{Q}}R(Q) \lesssim \varepsilon_N^2\log N.\]
        This means the approximation error contracts at the same convergence rates of posterior distributions.
        Thus, combining Theorems \ref{mainthm} and \ref{boundgam}, we have the following theorem.
        \begin{theorem}\label{finalthm}
            Consider $\theta_0$, $K$, $\alpha$, $p$, $\Pi'$, $\Pi_N(\cdot|D_N)$, $\varepsilon_N$, $\gamma_N$ be as defined in Theorem \ref{mainthm}.
      Suppose Condition \ref{con:condreg} holds for forward map $\mathcal{G}$,  separable normed linear subspace $(\mathcal{R},\Vert \cdot\Vert _{\mathcal{R}}) = (B_{pp}^b(\mathcal{Z}),\Vert \cdot\Vert _{B_{pp}^b(\mathcal{Z})})$ with some $b<\alpha-\frac{d}{p}$ and finite constants $C_U > 0$, $C_L > 0$ , $\kappa\geq0$, $\mu\geq 0$ and $l \geq 0$.
      Assume that $\alpha + \kappa \geq \frac{d(l+1)}{p}$ and $\theta_0\in B_{pp}^{\alpha}(\mathcal{Z})$ supported in the compact set $K$. Then, for variational posterior $\hat{Q}$ defined in \eqref{variationalposterior} with variational set $\mathcal{Q}_E$ in \eqref{GMFinf}, we have
      \begin{equation}\label{prate}
          P_{\theta_0}^{(N)}\hat{Q}\Vert \mathcal{G}(\theta)-\mathcal{G}(\theta_0)\Vert _{L^{2}_{\lambda}}^{\frac{2}{\mu+1}}\lesssim \varepsilon_N^{\frac{2}{\mu+1}}\log N.
      \end{equation}
      Moreover, assume that Condition \ref{con:condstab} also holds for $\mathcal{G}$, $\mathcal{R}$, function $F$ and finite constants $C_T > 0$, $\nu\geq 0$. Then, we further have
      \begin{equation}\label{rate}
          P_{\theta_0}^{(N)}\hat{Q}[F(\Vert f_{\theta} - f_{\theta_0}\Vert )]^{\frac{2}{\mu+\nu+1}}\lesssim \varepsilon_N^{\frac{2}{\mu+\nu+1}}\log N.
      \end{equation}
    \end{theorem}
    \par
    Our theory implies that convergence rates of the variational posteriors distributions for non-linear inverse problems can reach the same convergence rates of posterior distributions shown in \cite{IntroNonLinear_nickl2023bayesian,agapiou2024laplace} (with a logarithmic factor). As illustrated in Section 4 of \cite{agapiou2024laplace}, for the estimation of $\theta_0$ in Besov spaces $B_{pp}^{\alpha}$ with $p\in[1,2)$ in certain cases, the convergence rates obtained in \eqref{prate} and \eqref{rate} using Besov priors are better than the minmax rates using Gaussian priors.
    \begin{remark}
        The extra logarithmic factor arises in convergence rates
        due to the product form of the measures in $\mathcal{Q}_E$. {\color{black}For richer variational families, this logarithmic factor could be eliminated. For example, we define the variational family $\mathcal{Q}_E^b$ to be
\[ \mathcal{Q}_E^b =\set{Q=\bar{Q}|_A:  \bar{Q}\in \mathcal{Q}_E,\quad \mbox{any bounded set } A \mbox{ in } \mathcal{R} },\]
where $Q=\bar{Q}|_A$ means that for any measurable set $B$, $Q(B)= \bar{Q}(B\cap A)/\bar{Q}(A)$, that is, the restriction of the measure $\bar{Q}$ on a bounded set $A$. It has been proved that with this variational family, we can eliminate the $\log N$ factor when we estimate the upper bound of $\gamma_N^2$ (see Theorem 3.5 in \cite{zu2024consistencyvariationalbayesianinference} and its proof).} Besides, we note that (\ref{prate}) and (\ref{rate}) maintain their convergence rates when we extend our variational set from \textcolor{black}{$\mathcal{Q}_E$} to $\mathcal{Q}_{MF}$, which is defined by the push-forward of a standard mean-field family. The set $\mathcal{Q}_{MF}$ is explicitly given by
    \begin{equation*}
        \mathcal{Q}_{MF} = \Bigg\{Q = \tilde{Q}\circ\Psi^{-1} : d\tilde{Q}(\tilde{\theta}) = \prod_{l=-1}^{+\infty}\prod_{r\in R_l} d\tilde{Q}_{lr}(\tilde{\theta}_{lr}) \Bigg\},
    \end{equation*}
    which has been extensively utilized in various practical applications, as demonstrated in references \cite{IntroVI_zhang2018advances,IntroVI_blei2017variational,IntroVI_jia2021variational,Jia2023JMLR}.
\end{remark}
\subsection{Results for high-dimensional Besov priors} \label{High-dimensional Gaussian sieve priors}
Note that Theorem \ref{mainthm} is restricted to $p\in[1,2]$ because the decentering property Lemma \ref{lem:decenter} may not be established for the prior $\Pi'$ defined in Definition \ref{Def:BesovPrior} with $p>2$. To give consistency results similar to Theorem \ref{mainthm} for $\theta_0\in B_{pp}^{\alpha}$ with $p>2$, we consider high-dimensional Besov priors $\Pi_N$ which arise from the truncated Besov prior in Definition \ref{Def:TrunBesovPrior}.

The settings of high-dimensional Besov priors in this section follow those of rescaled Besov priors but here $\Pi_N$ is the truncated Besov prior in Definition \ref{Def:TrunBesovPrior}, and we repeat the settings here for clarity.
Let $\Pi'_J$ be the centred Besov probability measure $\Pi'_J$ defined in Definition \ref{Def:TrunBesovPrior}. For $\kappa$ in Condition \ref{con:condreg}, let $\Pi_N$ be the corresponding $B_{pp}^{\alpha}$-Besov prior from Definition \ref{Def:TrunBesovPrior}, with scaling constant
        \[\rho = (N\varepsilon_N^2)^{-\frac{1}{p}}, \qquad \varepsilon_N = N^{-\frac{\alpha+\kappa}{2\alpha+2\kappa+d}},\]
        where $N$ is the number of data points from the model \eqref{model}.
We define finite dimensional approximation of $\theta_0\in L^2_K(\mathcal{Z})$ to be
\begin{equation}
P_J(\theta_0) = \sum_{l=-1}^{J}\sum_{r\in R_l}\pdt{\theta_0}{\psi_{lr}}_{L^2(\mathcal{Z})}\psi_{lr}, \quad J \in \mathbb{N}.
\end{equation}
In order to make the $L^2$-error between $P_J(\theta_0)$ and $\theta_0$ dominated by the convergence rates of the variational posterior, it is necessary to let the truncation point $J = J_N$ diverge when $N\rightarrow+\infty$ (see Lemma \ref{0con}).
\par
In analogy to Theorem \ref{mainthm}, using high-dimensional Besov priors, we first derive a contraction rate involving the term $\gamma_N^2$ for $\theta_0\in B_{pp}^{\alpha}$ with $p\geq2$.
\begin{theorem} \label{mainthmsv}
Consider $\theta_0\in B_{pp}^{\alpha}(\mathcal{Z})$ supported in the compact set $K$ for some $\alpha>\frac{d}{p}$ and $p\geq 2$.
Suppose Condition \ref{con:condreg} holds for the forward map $\mathcal{G}$, separable normed linear subspace $(\mathcal{R}, \Vert \cdot\Vert _{\mathcal{R}}) = (B_{pp}^{b}(\mathcal{Z}), \Vert \cdot\Vert _{B_{pp}^{b}(\mathcal{Z})})$ with some $b \leq \alpha$ and finite constants $C_U > 0$, $C_L > 0$ , $\kappa\geq0$, $\mu\geq 0$ and $l \geq 0$.
Let probability measure $\Pi'_J$ be as in Definition \ref{Def:TrunBesovPrior}, and $J = J_N \in \mathbb{N}$ is such that $2^{J} \simeq N^{\frac{1}{2\alpha +2{\kappa} +d}}$. 
Denote by $\Pi_N$ the rescaled prior as in (\ref{rescaledprior}) with $\theta' \sim \Pi'_J$, and $\Pi_N(\cdot|(Y_i,X_i)_{i=1}^N) = \Pi_N(\cdot|D_N)$ is the corresponding posterior distribution in (\ref{Post}) arising from data in model (\ref{model}). Assume that $\alpha + {\kappa} \geq \frac{d(l+1)}{p}$. Then, for $\varepsilon_N=N^{-\frac{\alpha+{\kappa}}{2\alpha+2{\kappa}+d}}$, $\gamma_N^2 = \frac{1}{N}\mathop{\inf}_{Q \in \mathcal{Q}}P_{\theta_0}^{(N)}D(Q\Vert \Pi_N(\cdot|D_N))$, and variational posterior $\hat{Q}$ defined in (\ref{variationalposterior}), we have
  \begin{equation}
      P_{\theta_0}^{(N)}\hat{Q}\Vert \mathcal{G}(\theta)-\mathcal{G}(\theta_0)\Vert _{L^{2}_{\lambda}}^{\frac{2}{\mu+1}}\lesssim \varepsilon_N^{\frac{2}{\mu+1}} +\gamma^2_N \cdot \varepsilon_N^{-\frac{2\mu}{\mu+1}}.
  \end{equation}
  Moreover, assume that Condition \ref{con:condstab} also holds for $\mathcal{G}$, $\mathcal{R}$, the function $F$, and the finite constants $C_T > 0$, $\nu\geq 0$. Then, we further have
  \begin{equation}
      P_{\theta_0}^{(N)}\hat{Q}[F(\Vert f_{\theta} - f_{\theta_0}\Vert )]^{\frac{2}{\mu+\nu+1}}\lesssim \varepsilon_N^{\frac{2}{\mu+\nu+1}} +\gamma^2_N \cdot \varepsilon_N^{-\frac{2\mu+2\nu}{\mu+\nu+1}}.
  \end{equation}
\end{theorem}
We note that the sample of high-dimensional priors $\Pi_N$ is parameterized by the coefficients of the basis, that is,
\[\left\{\tilde{\theta} = (\tilde{\theta}_{lr})\in \mathbb{R}^{d_J}: r \in R_l, l \in \{-1,0,\dots,J\}\right\},\]
where $d_J:= \sum_{l=-1}^{J}\abs{R_l}\simeq 2^{Jd}$.
Thus, it is possible to find a variational posterior in a variational set containing finite-dimensional probability measures.
Define a p-exponential mean-field family $\tilde{\mathcal{Q}}_E^J(q)$ for $q\geq 1$ as
\begin{equation}\label{OGMF}
\Bigg\{Q = \mathop{\bigotimes}_{l=-1}^J\mathop{\bigotimes}_{r\in R_l}Exp(q;\mu_{lr},\sigma_{lr}^2):\mu_{lr} \in \mathbb{R}, \sigma_{lr}^2 \geq 0\Bigg\}.
\end{equation}
The map $\Psi_J: \mathbb{R}^{d_J} \mapsto L^2(\mathcal{Z})$ is defined by
\begin{equation}\label{pushforwardmap}
\Psi_J(\tilde{\theta}) = \sum_{l=-1}^J\sum_{r\in R_l}\tilde{\theta}_{lr}\chi\psi_{lr},\quad \forall\, \tilde{\theta} = (\tilde{\theta}_{lr})\in \mathbb{R}^{d_J}.
\end{equation}
Thus, our variational set $\mathcal{Q}_E^J(q)$ is obtained by the push-forward of $\tilde{\mathcal{Q}}_E^J(q)$ via $\Psi_J$, i.e.,
\begin{equation}\label{GMF}
\mathcal{Q}_E^J(q) = \Bigg\{Q = \tilde{Q}\circ\Psi_J^{-1} : \tilde{Q} \in \tilde{\mathcal{Q}}_E^J(q)\Bigg\}.
\end{equation}
We can find a probability measure $Q\in \mathcal{Q}_E^J(q)$ with $q \geq p$ such that $R(Q) \lesssim \varepsilon_N^2\log N$, which gives the convergence rate in the following theorem.
\begin{theorem}\label{boundgamfinite}
Suppose the forward map $\mathcal{G}$ satisfies Condition \ref{con:condreg} and the Besov prior $\Pi_N$, the true parameter $\theta_0$ are as defined in Theorem \ref{mainthmsv}. Then, there exists a probability measure $Q_N \in \mathcal{Q}_E^J(q)$ such that
        \begin{equation*}
            R(Q_N) \lesssim \varepsilon_N^2\log N
        \end{equation*}
        for $\varepsilon_N=N^{-\frac{\alpha+{\kappa}}{2\alpha+2{\kappa}+d}}$.
\end{theorem}
\par
With Theorem \ref{boundgamfinite}, it is easy to bound $\gamma_N^2$ by
$\gamma_N^2\leq \inf_{Q\in \mathcal{Q}_E^J(q)}R(Q) \leq R(Q_N).$
Then, Theorems \ref{mainthmsv} and \ref{boundgamfinite} together imply the convergence rate of the variational posterior from $\mathcal{Q}_E^J(q)$.
\begin{theorem}\label{finalthmsv}
    Consider $\theta_0$, $K$, $\alpha$, $p$, $\Pi'_J$, $\Pi_N(\cdot|D_N)$, $\varepsilon_N$, $\gamma_N$ be as defined in Theorem \ref{mainthmsv}.
      Suppose Condition \ref{con:condreg} holds for the forward map $\mathcal{G}$, separable normed linear subspace $(\mathcal{R}, \Vert \cdot\Vert _{\mathcal{R}}) = (B_{pp}^{b}(\mathcal{Z}), \Vert \cdot\Vert _{B_{pp}^{b}(\mathcal{Z})})$ with some $b \leq \alpha$ and finite constants $C_U > 0$, $C_L > 0$ , $\kappa\geq0$, $\mu\geq 0$ and $l \geq 0$.
Assume that $\alpha + {\kappa} \geq \frac{d(l+1)}{p}$ and $\theta_0\in B_{pp}^{\alpha}(\mathcal{Z})$ supported in the compact set $K$. Then, for variational posterior $\hat{Q}$ defined in (\ref{variationalposterior}) with variational set $Q_E^J(q)$ in (\ref{GMF}), we have
  \begin{equation*}
      P_{\theta_0}^{(N)}\hat{Q}\Vert \mathcal{G}(\theta)-\mathcal{G}(\theta_0)\Vert _{L^{2}_{\lambda}}^{\frac{2}{\mu+1}}\lesssim \varepsilon_N^{\frac{2}{\mu+1}}\log N.
  \end{equation*}
  Moreover, assume that Condition \ref{con:condstab} also holds for $\mathcal{G}$, $\mathcal{R}$, the function $F$ and the finite constants $C_T > 0$, $\nu\geq 0$. Then, we further have
  \begin{equation*}
      P_{\theta_0}^{(N)}\hat{Q}[F(\Vert f_{\theta} - f_{\theta_0}\Vert )]^{\frac{2}{\mu+\nu+1}}\lesssim \varepsilon_N^{\frac{2}{\mu+\nu+1}}\log N.
  \end{equation*}
\end{theorem}
\par
Compared to Theorem \ref{finalthm},  we observe that Conditions \ref{con:condreg} and \ref{con:condstab} are required to hold on $B_{pp}^{b}$ for some $b\leq \alpha$ rather than a stronger constraint $b<\alpha - \frac{d}{p}$. This relaxation of requirements stems from the regularity-enhancing property of truncated priors, which increases the reconstruction rate of the truth $f_0$, e.g., improvement of the upper bound of $s$ in Theorem \ref{mainthmDarcy}. 
\begin{remark}
Although Theorem \ref{finalthmsv} mainly addresses the case \( p \geq 2 \), convergence rates under high-dimensional priors remain attainable for \( p \in [1, 2) \) (see the proof of Theorem \ref{finalthmsv}). This extension requires strengthening the regularity requirement for \( \theta_0 \), specifically imposing a more critical regularity constraint \( \alpha_0 \geq \alpha + d/p - d/2 \). 
Failure to meet this condition creates a discrepancy between the \( L^2 \)-approximation error of \( P_J(\theta_0) \) for \( \theta_0 \) and the convergence rates of the variational posterior, arising from the absence of square-integrability (see Lemma \ref{0con}).
\end{remark}

\section{Contraction rate for two typical inverse problems}\label{sec:ApplicationSection}
In this section, we apply Theorem \ref{finalthm} to the Darcy flow problem and the inverse potential problem for a subdiffusion equation. In order to verify Conditions \ref{con:condreg} and \ref{con:condstab} for these problems, we use link functions satisfying specific properties as in \cite{IntroNonLinear_nickl2020convergence}. {\color{black} We note that the following results on specific nonlinear inverse problems can be directly extended to the case where $p >2$ using Theorem \ref{finalthmsv} with the truncated Besov prior.}
\subsection{Darcy flow problem}\label{SubsectionDarcyFlow}
For a bounded smooth domain $\mathcal{X} \subset \mathbb{R}^d$ ($d\in \mathbb{N}$) and a given source function $g \in C^{\infty}(\mathcal{X})$, we consider solutions $u=u_f$ to the Dirichlet boundary problem
\begin{equation}\label{Darcy}
    \left\{\begin{aligned}
    &\nabla\cdot(f \nabla u) = g \quad \mbox{on } \mathcal{X},\\
    &u = 0 \quad \mbox{on } \partial\mathcal{X}.
    \end{aligned}\right.  
\end{equation}
In this subsection, we will identify $f$ from the observation of $u_f$.
Assume the parameter $f \in \mathcal{F}^{\alpha}_{p,K_{\min}}$ for some $\alpha > 1 + d/2$, $K_{\min} \in [0,1)$, where the parameter space $\mathcal{F} := \mathcal{F}^{\alpha}_{p,K_{\min}}$ is defined as
\begin{equation}
\begin{aligned}
    \mathcal{F}^{\alpha}_{p,K_{\min}} =  \bigg\{ &f\in B_{pp}^{\alpha}(\mathcal{X}): f > K_{\min} \text{\ on\ } \mathcal{X}, f= 1 \text{\ on\ } \partial \mathcal{X},  \\ 
     &\qquad\quad\left. \frac{\partial^j f}{\partial n^j} = 0 \text{\ on\ } \partial \mathcal{X} \text{\ for\ } j= 1,\dots,\alpha -1 \right\}.
\end{aligned}  
\end{equation}
Here, the forward map $G$ is defined by 
\begin{equation*}
G : \mathcal{F}^{\alpha}_{p,K_{\min}} \rightarrow L^2_{\lambda}(\mathcal{X}),\qquad f \mapsto u_f,
\end{equation*}
where the probability measure $\lambda$ is chosen as the uniform distribution on $\mathcal{X}$. 
To build re-parametrisation of $\mathcal{F}_{\alpha,K_{\min}}$, we introduce the approach of using regular link functions $\Phi$ as in \cite{IntroNonLinear_nickl2020convergence}.
Define a function $\Phi$ that satisfies the following properties:
\begin{description}[labelindent = 0pt, leftmargin =*]
    \item[(i)] For given $K_{\min}> 0$, $\Phi: \mathbb{R}\rightarrow (K_{\min},\infty)$ is a smooth, strictly increasing bijective function such that $\Phi(0)=1$ and $\Phi'>0$ on $\mathbb{R}$;
    \item[(ii)] All derivatives of $\Phi$ are bounded, i.e., $\sup_{x\in\mathbb{R}}\abs{\Phi^{(k)}(x)} < \infty$ for $k\geq 1.$
\end{description}
An example of such a link function is given in Example B.1 of the Supplementary Materials \cite{zu2024consistencyvariationalbayesianinference}.
We set \[\Theta^{\alpha}_{p,K_{\min}} : = \set{\theta = \Phi^{-1}\circ f: f\in \mathcal{F}^{\alpha}_{p,K_{\min}} }.\] The reparametrized forward map $\mathcal{G}$ is then defined as
\begin{equation}\label{forwardmapDarcyflow}
\mathcal{G} : \Theta^{\alpha}_{p,K_{\min}} \rightarrow L^2_{\lambda}(\mathcal{X}), \qquad \theta \mapsto \mathcal{G}(\theta):=G(\Phi(\theta)).
\end{equation}
It can be verified through the properties of $\Phi$ (see Section 6 of \cite{IntroNonLinear_nickl2020convergence}) that
\begin{equation*}
\Theta^{\alpha}_{p,K_{\min}} = \bigg\{\theta \in B_{pp}^{\alpha}(\mathcal{X}): \frac{\partial^j \theta}{\partial n^j} = 0 \text{\ on\ } \partial \mathcal{X} \text{\ for\ } j=0,\dots,\alpha-1\bigg\}.
\end{equation*}
The reason why we use the link function $\Phi$ instead of the common choice $\Phi = \exp$ is that Conditions \ref{con:condreg} and \ref{con:condstab} require the polynomial growth in $\norm{\theta}_{\mathcal{R}}$ of those constants. If we use $\Phi = \exp$ as the link function, the polynomial growth is not satisfied.
\begin{theorem}\label{mainthmDarcy}
Let $d \in \mathbb{N}$, $p\in[1,2]$, $\alpha > (2 + 2d/p)\vee(4d/p-1)$ and ${\kappa}=1$. Consider the forward map $\mathcal{G}$ as in (\ref{forwardmapDarcyflow}). Let $\Pi'$, $\Pi_N(\cdot|D_N)$, $\mathcal{Q}_E$ and $\hat{Q}$ be as defined in Theorem \ref{finalthm}. Assume that $\theta_0\in B_{pp}^\alpha(\mathcal{X})$ is compactly supported on $K$. Then, for $\varepsilon_N=N^{-\frac{\alpha+{\kappa}}{2\alpha+2{\kappa}+d}}$, any $s$ such that $2<s-d/2<\alpha-2d/p$, we have
\begin{gather}\label{DracyGeneralerror}
    P_{\theta_0}^{(N)}\hat{Q} \Vert u_{f_{\theta}} - u_{f_0}\Vert_{L^2}^{\frac{2}{\mu+1}}\lesssim \varepsilon_N^{\frac{2}{\mu+1}}\log N,\\
\label{Dracyerror}
    P_{\theta_0}^{(N)}\hat{Q} \Vert f_{\theta} - f_{0}\Vert_{L^2}^{\frac{s+1}{s-1}\cdot\frac{2}{\mu+\nu+1}}\lesssim \varepsilon_N^{\frac{2}{\mu+\nu+1}} \log N
\end{gather}
with any $t$ such that $2<t-d/2<\alpha-2d/p$, $\mu = t^3 + t^2$, $\nu = \frac{(2s^2+1)(s+1)}{s-1}$. 
\end{theorem}
\par
We note that for the ``PDE-constrained regression'' problem of recovering $u_{f_0}$ in ``prediction'' loss, the convergence rate
obtained in \eqref{DracyGeneralerror} can be shown to be minimax optimal (up to a logarithmic factor) \cite[Section 2.3.2]{IntroNonLinear_giordano2020consistency}. For a smooth truth $f_0$, both of the rates obtained in \eqref{DracyGeneralerror} and \eqref{Dracyerror} approach the optimal rate $N^{-1/2}$ of finite-dimensional models as $\alpha \rightarrow +\infty$. In particular, our variational reconstruction rate of $f_0$ matches the true posterior contraction rates established in \cite{IntroNonLinear_nickl2023bayesian} for Gaussian process priors. However, the optimal reconstruction rate for the Darcy flow problem with general Besov regularity of the truth $f_0$ remains to be studied for future research.
\begin{remark}
For the inverse potential problem of the Schrödinger equations discussed in \cite{IntroNonLinear_nickl2020bernstein}, it is worth noting that the convergence rate of the variational posterior can also be obtained using the link function detailed in Section \ref{SubsectionSuddiffuion}. when $p=2$, the convergence rate towards the truth $f_0$ derived from Theorem \ref{finalthmsv} reaches the same rate $N^{-\frac{\alpha}{2\alpha + 4 + d}}$ (up to a logarithmic factor) as that proved to be minimax optimal in \cite{IntroNonLinear_nickl2020bernstein}. Since our results for Schr$\ddot{\mathrm{o}}$dinger equations can be obtained directly through a process similar to that used for the Darcy flow problem, using regularity and conditional stability estimates from \cite{IntroNonLinear_nickl2023bayesian}, we will not provide theorems and proofs here. Instead, we present our results on the inverse potential problem for a subdiffusion equation in Section \ref{SubsectionSuddiffuion}.
\par
\end{remark}
\subsection{Inverse potential problem for a subdiffusion equation}\label{SubsectionSuddiffuion}
Let domain $\Omega = (0,1)$ and we consider solutions $u(t)=u_{\beta,q}(t)$ to a subdiffusion equation with a non-zero Dirichlet boundary condition:
\begin{equation}\label{Fractional}
    \left\{\begin{aligned}
    &\partial_t^{\beta}u - \partial_{xx}u + qu = f \quad \mbox{in } \Omega \times (0,T],\\
    &u(0,t) = a_0, u(1,t) = a_1 \quad \mbox{on } (0,T],\\
    &u(0) = u_0 \quad \mbox{in } \Omega,
    \end{aligned}\right.  
\end{equation}
where $\beta \in (0,1)$ represents the fractional order, $T > 0$ stands for a fixed final time, $f>0$ is a specified source term, $u_0>0$ denotes given initial data, the non-negative function $q\in L^{\infty}(\Omega)$ refers to a spatially dependent potential, and $a_0$ and $a_1$ are positive constants. The notation $\partial^\beta_tu(t)$ denotes the Djrbashian--Caputo fractional derivative in time $t$ of order $\beta\in (0,1)$, 
\begin{equation}
\partial_t^{\beta}u(t) = \frac{1}{\Gamma(1-\beta)}\int^t_0(t-s)^{-\beta}u'(s)ds,
\end{equation}
where $\Gamma(x)$ is the Gamma function. 
For in-depth analysis of fractional differential equations and the Djrbashian-Caputo fractional derivative, please refer to references \cite{jin2021fractional,Jia2017JDE,Jia2018IPI,Jia2018JFA}.
In this section we consider the identification of the potential $q$ from the observation of $u(T)$. 
\par
For $\alpha \in \mathbb{N}$, we define the parameter space
\begin{equation}\label{FracParameterspcae}
\mathcal{F}^{\alpha}_{p,M_0} =\left\{ q\in B_{pp}^{\alpha} \cap \mathcal{I}: q\vert_{\partial\Omega}=1, \frac{\partial^jq}{\partial n^j}\Big|_{\partial\Omega} =0 \text{\ for\ } j=1,\dots,\alpha -1 \right\},
\end{equation}
where $\mathcal{I} = \{q \in L^{\infty}: 0< q < M_0\}$ for $M_0 >1$, and its subclasses
\[\mathcal{F}^{\alpha}_{p,M_0}(R) =\left\{ q\in \mathcal{F}_{p,M_0}^{\alpha}: \norm{q}_{B^{\alpha}_{pp}}\leq R\right\}, \quad R > 0.\]
We assume $u_0 \in B^{\alpha}_{pp}(\Omega)$, $f \in B^{\alpha}_{pp}(\Omega)$ with $u_0, f \geq L_0$ a.e. and $a_0, a_1 \geq L_0$ for $L_0>0$.
Here the forward map $G$ is defined by 
\begin{equation*}
G : \mathcal{F}^{\alpha}_{p,M_0} \rightarrow L^2_{\lambda}(\Omega),\qquad q \mapsto u_q(T),
\end{equation*}
where probability measure $\lambda$ is chosen as the uniform distribution on $\Omega$.
We use a link function $\Phi$ here to construct a reparametrization of $\mathcal{F}_{\alpha,M_0}$.
Define a link function $\Phi$ that satisfies the following properties:
\begin{description}[labelindent = 0pt, leftmargin =*]
    \item[(i)] For given $M_0> 1$, $\Phi: \mathbb{R}\rightarrow (0,M_0)$ is a smooth, strictly increasing bijective function such that $\Phi(0)=1$ and $\Phi'>0$ on $\mathbb{R}$;
    \item[(ii)] All derivatives of $\Phi$ are bounded, i.e.,
        $\sup_{x\in\mathbb{R}}\abs{\Phi^{(k)}(x)} < \infty,$ for $k\geq 1.$
\end{description}
One example to satisfy (i) and (ii) is the logistic function \cite{furuya2024consistency}:
\[\Phi(t) = \frac{M_0}{M_0+(M_0-1)(e^{-t}-1)}.\]
We set $\Theta^{\alpha}_{p,M_0} : = \left\{\theta = \Phi^{-1}\circ q: q\in \mathcal{F}^{\alpha}_{p,M_0} \right\}$. The reparametrized forward map $\mathcal{G}$ is then defined as
\begin{equation}\label{forwardmapFrac}
\mathcal{G} : \Theta^{\alpha}_{p,M_0} \rightarrow L^2_{\lambda}(\Omega), \qquad \theta \mapsto \mathcal{G}(\theta):=G(\Phi(\theta)).
\end{equation}
It is verified through the properties of $\Phi$ that
\begin{equation*}
\Theta^{\alpha}_{p,M_0} = \left\{\theta \in B_{pp}^{\alpha}: \frac{\partial^j \theta}{\partial n^j} = 0 \text{\ on\ } \partial \Omega \text{\ for\ } j=0,\dots,\alpha-1\right\}.
\end{equation*}
\par
\begin{theorem}\label{mainthmFrac}
Let $d = 1$, $p \in [1,2]$, $\alpha > 2+2d/p$ and ${\kappa}= 2$. Consider the forward map $\mathcal{G}$ as in (\ref{forwardmapFrac}) with terminal time $T\geq T_0$ where $T_0$ is large enough. Let $\Pi'$, $\Pi_N(\cdot|D_N)$, $\mathcal{Q}_E$ and $\hat{Q}$ be as defined in Theorem \ref{finalthm}. Assume that $\theta_0\in B_{pp}^\alpha(\Omega)$ is compactly supported on $K$. Then, for $\varepsilon_N=N^{-\frac{\alpha+{\kappa}}{2\alpha+2{\kappa}+d}}$, any integer $s$ such that $0\leq s<\alpha+d/2-2d/p$, we have
\begin{gather}\label{FracGeneralerror}
   P_{\theta_0}^{(N)}\hat{Q} \Vert u_{q_{\theta}}(T) - u_{q_{0}}(T)\Vert_{L^2}^{\frac{2}{\mu+1}}\lesssim \varepsilon_N^{\frac{2}{\mu+1}}\log N, \\
\label{Fracerror}
    P_{\theta_0}^{(N)}\hat{Q} \Vert q_{\theta} - q_{0}\Vert_{L^2}^{\frac{2+s}{s}\cdot\frac{2}{\mu+\nu+1}}\lesssim \varepsilon_N^{\frac{2}{\mu+\nu+1}}\log N
\end{gather}
with $\mu = 0$, $\nu = 2+4s$.
\end{theorem}
\par
It is observed that for the ``PDE-constrained regression'' problem of recovering $u_{q_{0}}(T)$ in ``prediction'' loss, the convergence rate found in \eqref{FracGeneralerror} will be demonstrated to be minimax optimal (up to a logarithmic factor), as evidenced by Theorem \ref{Fracminmax} given below. The reconstruction rate of {\color{black}$q_0$} obtained in \eqref{Fracerror} increases with $p$ because $s/(2+s)$ can have a higher value. However, the optimal reconstruction rate with general Besov regularity of the truth {\color{black}$q_0$} remains to be studied for future research. For a smooth truth $q_0$, the rates obtained in \eqref{FracGeneralerror} and \eqref{Fracerror} both approach the optimal rate $N^{-1/2}$ of finite-dimensional models as $\alpha \rightarrow +\infty$.
\begin{remark}
    The convergence rates with high-dimensional priors can also be obtained for this inverse problem using Theorem \ref{finalthmsv}, which improves the rates of recovering {\color{black}$q_0$}, since the support of truncated priors has higher regularity. The convergence rates of recovering {\color{black}$q_0$} with high-dimensional priors can achieve the minimax optimal rate when $p=2$ (see Theorem 4.4 in \cite{zu2024consistencyvariationalbayesianinference}).
\end{remark}
\begin{theorem}\label{Fracminmax}
For $M_0>1$, $p\in[1,2]$, $\alpha \in \mathbb{N}$, $q\in \mathcal{F}^{\alpha}_{p,M_0}$ , consider the solution $u_q(t)$ of the problem \eqref{Fractional}. Then there exist fixed $T_0>0$ and a finite constant $C>0$ such that for $N$ large enough, the terminal time $T\geq T_0$ and any $\eta>0$,
    \begin{equation*}
         \inf_{\tilde{u}_N}\sup_{q_0\in\mathcal{F}^{\alpha}_{p,M_0}(R)}P^{(N)}_{\theta_0}\hat{Q}\norm{\tilde{u}_N-u_{q_0}(T)}^{\eta}_{L^2(\Omega)}\geq C N^{-\frac{\alpha+2}{2\alpha+4+1}\cdot\eta},
     \end{equation*}
      where $\theta_0 = \Phi^{-1}(q_0)$ and the infimum ranges over all measurable functions $\Tilde{u}_N = \Tilde{u}_N(\theta)$ that take value in $L^2(\Omega)$ with $\theta$ from the variational posterior $\hat{Q}$ defined in Theorem \ref{mainthmFrac}. 
\end{theorem}

{\color{black}\section{Conclusion} 
\subsection{Summary of findings}
We establish convergence rates for variational posteriors in nonlinear inverse problems under Besov priors induced by $p$-exponential wavelet coefficients, thereby extending the Gaussian-prior framework to this broader class. Under suitable regularity and stability conditions on the forward map, the variational posterior contracts toward the true parameter at the same rate as the posterior. We also develop new technical tools for Besov priors to control prior mass near the truth. Furthermore, we propose a $p$-exponential mean-field variational family and construct a distribution achieving approximation error $\gamma_N^2 \lesssim \varepsilon_N^2 \log N$, such that the variational posterior inherits the posterior contraction rate. The theory is applied to the Darcy flow and subdiffusion inverse problems, yielding minimax-optimal rates for the ``PDE-constrained regression problems''.

\subsection{Limitations and future directions}
Despite these advances, several limitations of the current work and directions for future research should be acknowledged. By addressing the limitations outlined below and pursuing the proposed directions, we hope to further advance the theory of variational inference for Bayesian inverse problems.

1. Assumptions on the forward map. Conditions \ref{con:condreg} and \ref{con:condstab} require polynomial growth bounds on the uniform norm, Lipschitz constant, and stability constant of $\mathcal{G}$. For the PDEs considered, these bounds are verified using specially designed link functions; the common exponential link function $e^x$ would instead yield exponential growth. Severely ill‑posed problems (often with logarithmic stability) may also lead to exponential bounds. One could enforce constant bounds by restricting the prior to a ball in Besov space \cite{fan_contraction_2026}, but this requires prior knowledge of the Besov norm of the true parameter, which may therefore be impractical. Relaxing these conditions remains an open problem.

2. Regularity requirements on the truth. The contraction results require the true parameter $\theta_0$ to belong to $B_{pp}^\alpha(\mathcal{Z})$, with $\alpha$ sufficiently large relative to $d$ and $p$, yet parameters with low regularity are common in practice. A recent study reduces the regularity requirement at the cost of slower contraction rates, so such a result may not be minimax optimal \cite{fan_contraction_2026}. How to bridge this gap remains an open question.

3. Adaptive variational inference. The rescaled Besov prior uses a scaling $\rho = (N\varepsilon_N^2)^{-1/p}$ with $\varepsilon_N = N^{-\frac{\alpha+\kappa}{2\alpha+2\kappa+d}}$, so the prior depends explicitly on sample size and on the regularity $\alpha$ and smoothing index $\kappa$. We hope to adaptively choose $\alpha$ and preserve the optimal contraction rate . Exact Bayesian methods have attained optimal adaptive contraction rates for linear problems \cite{szabo_frequentist_2015,knapik_bayes_2016}, whereas this remains an open problem for general nonlinear inverse problems. Variational Bayes is gradually developing an adaptive theory for linear problems, as demonstrated by recent advances in hierarchical variational Bayes \cite{nieman_adaptive_2025}. However, this remains an open problem for general nonlinear inverse problems.

4. General Banach spaces.  A recent study has developed the contraction theory of the Bayesian posterior for linear inverse problems on general Banach spaces \cite{chen_2024_posterior}. It may also be possible to extend our theory to certain abstract evolution equations defined on general Banach spaces.}

\section{Proofs of results for inverse problems} \label{sec:proofIP}
\subsection{Relations between information distances and the $L^2$ norm}
In order to give our results about the variational posterior, we introduce some information distances and relate these distances on the laws $\{P_{\theta} : \theta \in \Theta\}$ to the forward map $\mathcal{G}({\theta})$.
With $\gamma>0$ and $\gamma \neq 1$, the $\gamma$-R$\acute{\mathrm{e}}$nyi divergence between two probability measures $P_1$ and $P_2$ is defined as
\begin{equation*}
    D_{\gamma}(P_1\Vert P_2) = \left\{ \begin{aligned}
                 &\frac{1}{\rho - 1} \log \int \bbra{\frac{dP_1}{dP_2}}^{\rho-1} dP_1  &\text{if}\ P_1 \ll P_2,\\
                 &+ \infty  &\text{otherwise}.
            \end{aligned} \right.
\end{equation*}
When $\gamma \rightarrow 1$, the R$\acute{\mathrm{e}}$nyi divergence converges to the Kullback–Leibler (KL) divergence, defined as
\begin{equation*}
    D(P_1\Vert P_2) = \left\{ \begin{aligned}
                &\int \log\bbra{\frac{dP_1}{dP_2}} dP_1 & \text{if}\ P_1 \ll P_2,\\
                 &+ \infty & \text{otherwise}.
            \end{aligned} \right.
\end{equation*}
Moreover, the Rényi divergence $D_{\gamma}(P_1||P_2)$ is a non-decreasing function of $\gamma$, which particularly gives
$D(P_1\Vert P_2) \leq D_{2}(P_1\Vert P_2).$
The Hellinger distance $h$ is defined as
\begin{equation*}
    h^2(P_1,P_2) = \frac{1}{2}\int\bbra{\sqrt{dP_1}-\sqrt{dP_2}}^2.
\end{equation*}
\par
The following proposition from \cite{zu2024consistencyvariationalbayesianinference} relates these information distances on the laws $\{P_{\theta} : \theta \in \Theta\}$ to the $L^2_{\lambda}(\mathcal{X},V)$ norm, assuming $\mathcal{G}(\theta)$ are uniformly bounded by a constant $U$ that may depend on $\Theta$.
\begin{proposition}\label{le2.1}
    Suppose that for a subset $\Theta \subset L^2(\mathcal{Z},\mathbb{R})$ and some finite constant $U = U_{\mathcal{G},\Theta} > 0$, we have
    \[\mathop{\sup}_{\theta \in \Theta}\Vert\mathcal{G}(\theta)\Vert_{\infty} \leq U.\]
    For the model density from (\ref{modeldensity}), we have for every $\theta_1,\theta_2 \in \Theta$,
    \begin{equation}\label{D2}
        D_2(P_{\theta_1}\Vert P_{\theta_2}) \leq e^{4U^2} \norm{\mathcal{G}(\theta_1)-\mathcal{G}(\theta_2)}_{L^2_{\lambda}(\mathcal{X},V)}^2,
    \end{equation}
    \begin{equation}\label{KL}
        D(P_{\theta_1}\Vert P_{\theta_2}) = \frac{1}{2} \norm{\mathcal{G}(\theta_1)-\mathcal{G}(\theta_2)}_{L^2_{\lambda}(\mathcal{X},V)}^2,
    \end{equation}
    and
    \begin{equation}\label{hellinger}
        C_U\Vert\mathcal{G}(\theta_1)-\mathcal{G}(\theta_2)\Vert_{L^2_{\lambda}(\mathcal{X},V)}^2 \leq h^2(p_{\theta_1},p_{\theta_2}) \leq \frac{1}{4}\Vert\mathcal{G}(\theta_1)-\mathcal{G}(\theta_2)\Vert_{L^2_{\lambda}(\mathcal{X},V)}^2
    \end{equation}
    where 
    \[C_U = \frac{1-e^{-U^2/2}}{2U^2}.\]
\end{proposition}
\subsection{Contraction rates for rescaled Besov priors}
 \textcolor{black}{To deduce the contraction result, We verify the three conditions formulated in \cite[Theorem 2.1]{zhang2020convergence}.  Specifically, it is sufficient to prove that the following three conditions hold for a loss function $L(\cdot,\cdot)$ and constants $C,C_1,C_2,C_3 > 0$ with $C > C_2 +C_3 +2$:
\begin{itemize}
    \item[(C1)] For any $\varepsilon> \varepsilon_N$,there exists a set $\Theta_N(\varepsilon)$ and a testing function $\Psi_n$, such that
    \begin{equation*}
            P^{(N)}_{\theta_0}\Psi_N+\sup_{\substack{\theta \in \Theta_N(\varepsilon)  \\
                 L(P^{(N)}_{\theta},P^{(N)}_{\theta_0}) \geq N \varepsilon^2}}
            P^{(N)}_{\theta}(1-\Psi_N)\leq \exp\{-C N \varepsilon^2\}
        \end{equation*}
    \item[(C2)] For any $\varepsilon> \varepsilon_N$, the set $\Theta_N(\varepsilon)$ in (C1) satisfies $\Pi_N(\Theta_N(\varepsilon)^c)\le e^{-C N \varepsilon^2}$;
    \item[(C3)] For some constant $\gamma > 1$, $\Pi_N(\theta : D_{\gamma}(P^{(N)}_{\theta_0}\Vert P^{(N)}_{\theta})\leq C_3N\varepsilon_N^2)\geq e^{-C_2 N \varepsilon^2}$.
\end{itemize}}
\begin{proof}[Proof of Theorem \ref{mainthm}]
We are going to verify the three condition in \cite[Theorem 2.1]{zhang2020convergence}. Steps (i) to (iii) below verify conditions (C1) to (C3) directly.
\par
        We denote $U_{\mathcal{G}}(M),L_{\mathcal{G}}(M),T_{\mathcal{G}}(M)$ by $C_UM^p,C_LM^l,C_TM^q$ respectively.
        Set \[H_N(\varepsilon) = \set{\theta = \theta_1 + \theta_2: \norm{\theta_1}_{(H^{k})^*}\leq M^l\varepsilon_N/L_{\mathcal{G}}(Mr_N(\varepsilon)), \Vert \theta_2 \Vert_{B_{pp}^{\alpha}}\leq M(\varepsilon/\varepsilon_{N})^{\frac{2}{p}}}\] with $r_N(\varepsilon)=(\varepsilon/\varepsilon_N)^{\frac{2}{p}}$. We further define $\Theta_N(\varepsilon) = H_N(\varepsilon) \cap B_{\mathcal{R}}(Mr_N(\varepsilon)) \cap \mathrm{span}\{\chi\psi_{lr}\}_{l=-1}^{+\infty}$
            for some M large enough, $r\in R_l$ and any $\varepsilon > \varepsilon_N$. We denote $\mathcal{R}'$ by $\mathcal{R}$ with domain replaced by $\mathbb{R}^d$.
        \par
        (i) For (C1), we follow the method used in the proof of \cite[Theorem 7.1.4]{gin2015mathematical}. Let 
        \[S_j = \{\theta \in \Theta_N(\varepsilon):4j \bar{m} \varepsilon \leq h(p_{\theta},p_{\theta_0})< 4(j+1)\bar{m}\varepsilon\}, \quad j \in \mathbb{N}, \quad \bar{m} > 0.\]
        We see that $S_j \subseteq \Theta_N(\varepsilon),$
        so it is sufficient to consider the metric entropy of $\Theta_N(\varepsilon)$.
        Here we introduce the (semi-) metric 
        $d_{\mathcal{G}}({\theta_1},{\theta_2}) := \Vert \mathcal{G}(\theta_1)-\mathcal{G}(\theta_2)\Vert _{L^2_{\lambda}(\mathcal{X},V)}.$
        Using Proposition \ref{le2.1}, formula (4.184) in \cite{gin2015mathematical} and Lipschitz condition of $\mathcal{G}$, 
        we have
        \begin{align*}
            \log N(\Theta_N(\varepsilon),h,j\bar{m}\varepsilon) &\leq \log N(\Theta_N(\varepsilon),h,\bar{m}\varepsilon)\leq \log N(\Theta_N(\varepsilon),d_{\mathcal{G}},2\bar{m}\varepsilon)\\
            &\leq \log N\Big(\Theta_N(\varepsilon),\norm{\ \cdot \ }_{(H^{\kappa})^*},\frac{2\bar{m}\varepsilon}{C_L(Mr_N(\varepsilon))^l}\Big).
        \end{align*}
        Because $\Theta_N(\varepsilon) \subset \mathrm{span}\{\chi\psi_{lr}\}$, any $f_i  \in \Theta_N(\varepsilon)$ can be represented as 
        \[ f_i=\sum_{l=-1}^{+\infty}\sum_{r\in R_l}f_{i,lr}\chi\psi_{lr}.\]
        Therefore, for any $f_i,f_j\in\Theta_N(\varepsilon)$,
        \begin{align*}
            \norm{f_i-f_j}_{(H^{\kappa}(\mathcal{Z}))^*}=&\norm{\sum_{l=-1}^{+\infty}\sum_{r\in R_l}(f_{i,lr}-f_{j,lr})\chi\psi_{lr}}_{(H^{\kappa}(\mathcal{Z}))^*} \\= &\sup_{\norm{\phi}_{H^{\kappa}(\mathcal{Z})}\leq 1}\pdt{\sum_{l=-1}^{+\infty}\sum_{r\in R_l}(f_{i,lr}-f_{j,lr})\psi_{lr}}{\chi\phi}_{L^2(\mathcal{Z})} \\
     \lesssim &\sup_{\norm{\phi}_{H_{K'}^{\kappa}(\mathcal{Z})}\leq 1}\pdt{\sum_{l=-1}^{+\infty}\sum_{r\in R_l}(f_{i,lr}-f_{j,lr})\psi_{lr}}{\phi}_{L^2(\mathcal{Z})}\\
     \lesssim & \norm{f_i-f_j}_{H^{{-\kappa}}(\mathcal{Z})},
        \end{align*}
        which indicates that
        \[ \log N(\Theta_N(\varepsilon),h,j\bar{m}\varepsilon) \leq \log N\Big(\Theta_N(\varepsilon),\Vert\cdot \Vert_{H^{-\kappa}(\mathcal{Z})},\frac{2\bar{m}\varepsilon}{C_L(Mr_N(\varepsilon))^l}\Big).\]
        The definition of $H_N(\varepsilon)$ implies that, for $\bar{m} \geq M^l$, a $\frac{\bar{m}\varepsilon}{C_L(Mr_N(\varepsilon))^l}$-covering in $\norm{\ \cdot \ }_{H^{-\kappa}(\mathcal{Z})}$ of $B_{pp}^{\alpha}(M(\varepsilon/\varepsilon_{N})^{\frac{2}{p}})$ is a $\frac{2\bar{m}\varepsilon}{C_L(Mr_N(\varepsilon))^l}$-covering in $\norm{\ \cdot \ }_{H^{-\kappa}(\mathcal{Z})}$ of $\Theta_N(\varepsilon)$. 
        It is sufficient to consider \[\log N\Big(B_{pp}^{\alpha}(M(\varepsilon/\varepsilon_{N})^{\frac{2}{p}}),\norm{\ \cdot \ }_{H^{-\kappa}(\mathcal{Z})},\frac{\bar{m}\varepsilon}{C_L(Mr_N(\varepsilon))^l}\Big).\]
        By Theorem 2 in Section 3.3.3 and Remark
 1 in Section 1.3.1 of \cite{triebel1996function}, we have
        \begin{equation}\label{entropy}
            \log N(B_{pp}^{\alpha}(r), \norm{\ \cdot \ }_{H^{-\kappa}(\mathcal{Z})}, \delta) \leq C_E\left(\frac{r}{\delta}\right)^{d/(\alpha + {\kappa})}, \quad r,\delta > 0
        \end{equation}
        with some constant $C_E$.
        For $\bar{m} = C_LM^{l+1}$ and $\varepsilon > \varepsilon_N$, we deduced that
            \begin{align} 
            &\log N(B_{pp}^{\alpha}(M(\varepsilon/\varepsilon_{N})^{\frac{2}{p}}),\norm{\ \cdot \ }_{H^{-\kappa}(\mathcal{Z})},\frac{\bar{m}\varepsilon}{C_L(Mr_N(\varepsilon))^l}) \nonumber \\
            \leq &C_E\left(\frac{M C_L(Mr_N(\varepsilon))^l}{\bar{m}\varepsilon}\cdot\big(\frac{\varepsilon}{\varepsilon_N}\big)^{\frac{2}{p}}\right)^{d/(\alpha +{\kappa})}\label{entropybound}\\
            \leq &C_E\left(\frac{M^{l+1}C_{L}}{\bar{m}}\right)^{d/(\alpha+{\kappa})} \cdot (\frac{\varepsilon}{\varepsilon_N})^{\frac{2d(l+1)}{p(\alpha+\kappa)}}\cdot N\varepsilon_N^2 
            \leq C_EN\varepsilon^2,\nonumber
            \end{align}         
        where we used $\varepsilon/\varepsilon_N>1$, $d(l+1) \leq p(\alpha+\kappa) $, and $\varepsilon_N^{-d/(\alpha+\kappa)} = N\varepsilon_N^2$.
        Thus, we have 
        \[N(S_j,h,j\bar{m}\varepsilon)\leq N(\varepsilon) := \exp\{C_EN\varepsilon^2\}.\]
        Choose a minimum finite set $S_j^{'}$ of points in each set $S_j$
        such that every $\theta \in S_j$ is within Hellinger distance $j\bar{m}\varepsilon$ of at least one of these points. By metric entropy bound above, for $j$ fixed, there are at most $N(j\varepsilon)$ such points $\theta_{jl} \in S_j^{'}$, and from \cite[Corollary 7.1.3]{gin2015mathematical} for each $\theta_{jl}$ there exists a test $\Psi_{N,jl}$ such that
            \[P^{(N)}_{\theta_0}\Psi_{N,jl} \leq e^{-C_tNj^2\bar{m}^2\varepsilon^2}, \quad \mathop{\sup}_{\theta \in S_j, h(p_{\theta},p_{\theta_{jl}})<j\bar{m}\varepsilon} P^{(N)}_{\theta}(1-\Psi_{N,jl}) \leq e^{-C_tNj^2\bar{m}^2\varepsilon^2}\]
        for some universal constant $C_t > 0$. Let $\Psi_N = \mathop{\max}_{j,l}\Psi_{N,jl}$. Then, we have
        \begin{align}
            P^{(N)}_{\theta_0}\Psi_N & \leq P^{(N)}_{\theta_0}(\sum_j\sum_l\Psi_{N,jl}) \leq \sum_j\sum_l\exp\{-C_tNj^2\bar{m}^2\varepsilon^2\} \nonumber \\
            &\leq N(\varepsilon)\sum_j  \exp\{-C_tNj^2\bar{m}^2\varepsilon^2\}  \label{test1}\\
            &\leq \frac{1}{1-\exp\{-(C_t\bar{m}^2-C_E)\}}\exp\{-(C_t\bar{m}^2-C_E)N\varepsilon^2\}  \nonumber\\
            &\leq \exp\{-CN\varepsilon^2\} \nonumber
        \end{align}
        and
        \begin{equation}\label{test}
            \mathop{\sup}_{\begin{array}{c}
                 \theta \in \Theta_N(\varepsilon)  \\
                 h(p_{\theta},p_{\theta_0}) \geq 4\bar{m}\varepsilon
            \end{array}}
            P^{(N)}_{\theta}(1-\Psi_N) = \mathop{\sup}_{\theta \in \cup_j S_j}P^{(N)}_{\theta}(1-\Psi_N)\leq \exp\{-C N \varepsilon^2\}
        \end{equation}
        for any $C>0$ when $M$ is large enough.
        Using Proposition \ref{le2.1}, Conditions \ref{con:condreg} and \ref{con:condstab}, 
        we have the following inequality for $\theta \in \Theta_N(\varepsilon)$ and constants $C_{U},C_{T}$ from conditions (\ref{bound}) and (\ref{stab}):
        \begin{align*}
            h(p_{\theta},p_{\theta_0}) 
        &\geq \frac{1}{2C_{U}\cdot(M\varepsilon/\varepsilon_N)^{\mu}}\Vert \mathcal{G}(\theta)-\mathcal{G}(\theta_0)\Vert _{L^{2}_{\lambda}(\mathcal{X},V)}\\
        &\geq \frac{1}{2C_{U}C_{T}\cdot(M\varepsilon/\varepsilon_N)^{\mu+\nu}}F(\Vert f_{\theta} - f_{\theta_0}\Vert ).
        \end{align*}
        From the inequality and direct calculations, we note that the set
        \[\{\theta \in \Theta_N(\varepsilon) : h(p_{\theta},p_{\theta_0}) \geq 4\bar{m}\varepsilon\}\]
        contains
        \[
        \{\theta \in \Theta_N(\varepsilon) : N\varepsilon_N^{\frac{2\mu}{\mu+1}}\Vert \mathcal{G}(\theta)-\mathcal{G}(\theta_0)\Vert _{L^{2}_{\lambda}}^{\frac{2}{\mu+1}} \geq \tilde{C}_1 N \varepsilon^2 \}\]
        and
        \[
        \{\theta \in \Theta_N(\varepsilon) : N\varepsilon_N^{\frac{2\mu+2\nu}{\mu+\nu+1}}[F(\Vert f_{\theta} - f_{\theta_0}\Vert )]^{\frac{2}{\mu+\nu+1}} \geq C_1 N \varepsilon^2 \}\]
        for $\tilde{C}_1 = (8\bar{m}C_UM^{\mu})^{\frac{2}{\mu+1}}$ and $C_1 = (8\bar{m}C_U C_TM^{\mu+\nu})^{\frac{2}{\mu+\nu+1}}$. Combined with (\ref{test}), we have 
        \begin{equation}
            \sup_{\substack{\theta \in \Theta_N(\varepsilon)  \\
                 L(P^{(N)}_{\theta},P^{(N)}_{\theta_0}) \geq N \varepsilon^2}}
            P^{(N)}_{\theta}(1-\Psi_N)\leq \exp\{-C N \varepsilon^2\}
        \end{equation}
        with  
        \[L(P^{(N)}_{\theta},P^{(N)}_{\theta_0}) := N\varepsilon_N^{\frac{2\mu+2\nu}{v+1}}[F(\Vert f_{\theta} - f_{\theta_0}\Vert )]^{\frac{2}{\mu+\nu+1}}/C_1 \text{ or } N\varepsilon_N^{\frac{2p}{\mu+1}}\Vert \mathcal{G}(\theta)-\mathcal{G}(\theta_0)\Vert _{L^{2}_{\lambda}}^{\frac{2}{\mu+1}}/\tilde{C}_1.\]
        Thus, we have confirmed (C1) with the statement provided above.

        (ii) By the definition of $\Theta_N(\varepsilon)$, we deduce that 
        \begin{align*}
            \Pi_N(\Theta_N(\varepsilon)^c) &= \Pi_N(H_N(\varepsilon)^c \cup B_{\mathcal{R}}(Mr_{N}(\varepsilon))^c) \\
            & \leq \Pi_N(H_N(\varepsilon)^c) + \Pi_N(B_{\mathcal{R}}(Mr_N(\varepsilon))^c).
        \end{align*}
        Using Lemma \ref{lem:concentration}, we have
        \begin{align*}
            &\Pi_N(\Vert \theta\Vert _{\mathcal{R}}>M(\varepsilon/\varepsilon_N)^{\frac{2}{p}}) \leq \Pi'(\Vert \theta'\Vert _{\mathcal{R}'}>cM(N\varepsilon^2)^{\frac{1}{p}})\leq c_1e^{-c_2M^pN\varepsilon^2} \leq \frac{1}{2}e^{-CN\varepsilon^2}
        \end{align*}
        for the constant $C$ in (i) and $M$ large enough. Then it is sufficient to prove 
        \[\Pi_N(H_N(\varepsilon)) \geq 1 - \exp\{-BN\varepsilon^2\} \geq 1 - \frac{1}{2}\exp\{-CN\varepsilon^2\}\]
        for $B = C+2$.
        By the definition of $\Pi_N$ and $\Pi'$, there exists a small enough constant $c>0$ such that
        \begin{align*}
            \Pi_N(H_N(\varepsilon)) \geq &\Pi'\bigg(\theta : \chi\theta = \chi\theta_1 + \chi\theta_2, \norm{\chi\theta_1}_{(H^{\kappa}(\mathcal{Z}))^*}\leq \frac{M^l\varepsilon_N(N\varepsilon_N^2)^\frac{1}{p}}{L_{\mathcal{G}}(Mr_N(\varepsilon))},\\
            &\qquad\qquad\qquad\qquad\qquad\qquad\qquad\norm{\chi\theta_2}_{B_{pp}^{\alpha}(\mathcal{Z})}\leq (MN\varepsilon^2)^\frac{1}{p}\bigg)\\
            \geq &\Pi'\bigg(\theta = \theta_1 + \theta_2: \norm{\theta_1}_{(H^{\kappa}(\mathbb{R}^d))^*}\leq c\frac{M^l\varepsilon_N(N\varepsilon_N^2)^\frac{1}{p}}{L_{\mathcal{G}}(Mr_N(\varepsilon))},\\
            &\quad \norm{\theta_2}_{B_{pp}^{\alpha}({\mathbb{R}^d})}\leq c(MN\varepsilon^2)^\frac{1}{p}, \theta_i\in \mathrm{span}\{\psi_{lr}\},i=1,2\bigg).
        \end{align*}
        Following the proof of Lemma B.3 in \cite{agapiou2024laplaceSupp}, we deduce that for some small enough $c'>0$, $\Pi_N(H_N(\varepsilon))$ is bounded below by
        \begin{align*}
            &\Pi'\bigg(\theta = \theta_1 + \theta_2 +\theta_3: \norm{\theta_1}_{(H^{\kappa}({\mathbb{R}^d}))^*}\leq \frac{cM^l\varepsilon_N(N\varepsilon_N^2)^\frac{1}{p}}{2L_{\mathcal{G}}(Mr_N(\varepsilon))}, \norm{\theta_2}_{B_{pp}^{\alpha}({\mathbb{R}^d})}\leq \frac{c}{2}(MN\varepsilon^2)^\frac{1}{p},\\
            & \qquad \qquad \qquad \qquad \qquad \qquad \norm{\theta_3}_{H^{\alpha+\frac{d}{2}-\frac{d}{p}}({\mathbb{R}^d})} \leq \frac{c'}{2}\sqrt{M}\sqrt{N}\varepsilon, \theta_i\in \mathrm{span}\{\psi_{lr}\},i=1,2,3\bigg).
        \end{align*}
        
        Then, Lemma \ref{lem:twolevel} with $A=\{\theta\in \mathrm{span}\{\psi_{lr}\}: \norm{\theta}_{(H^{\kappa}({\mathbb{R}^d})^*)}\leq \frac{cM^l\varepsilon_N(N\varepsilon_N^2)^\frac{1}{p}}{2L_{\mathcal{G}}(Mr_N(\varepsilon))}\}$ and $r=c''(MN\varepsilon^2)^\frac{1}{p}$ for some small enough $c''$ leads to
        \begin{equation}\label{pbabound}
            \Pi_N(H_N(\varepsilon)) \geq 1-\frac{1}{\Pi'(A)}\mexp{-\frac{c''M}{\Lambda}N\varepsilon^2},
        \end{equation}
        where $\Lambda$ is a universal constant.
        Next, we give a lower bound of $\Pi'(A)$. By the definition of $A$ together with Lemma \ref{lem:smallball}, we have 
        \begin{align*}
        -\log\Pi'(A) 
            &\leq c_0({cM^l\varepsilon_N(N\varepsilon_N^2)^\frac{1}{p}}/{2L_{\mathcal{G}}(Mr_N(\varepsilon))})^{-\frac{pd}{p(\alpha+{\kappa})-d}}\\
            &\leq c_0[M^l/L_{\mathcal{G}}(M\varepsilon/\varepsilon_N)]^{-\frac{pd}{p(\alpha+{\kappa})-d}}N\varepsilon_N^2,
        \end{align*}
        for some fixed $c_0 > 0$.
        Combined with Lipschitz condition (\ref{lip}) and $p(\alpha +\kappa) \geq (l+1)d$, we deduced that 
        \[-\log\Pi'(A) \leq c_0{C_L}^{-\frac{pd}{p(\alpha+{\kappa})-d}}\cdot \bigg(\frac{\varepsilon}{\varepsilon_N}\bigg)^{\frac{2ld}{p(\alpha+\kappa)-d}}\cdot N\varepsilon_N^2\leq c_0{C_L}^{-\frac{pd}{p(\alpha+{\kappa})-d}}\cdot N\varepsilon_N^2.\]
        Let $\tilde{C} = c_0{C_L}^{-\frac{pd}{p(\alpha+{\kappa})-d}}$ and we have
            $\Pi'(A) \geq \mexp{-\tilde{C}N\varepsilon^2}.$
        Using the last inequality and \eqref{pbabound} with $M$ large enough, we obtain
        \[\Pi_N(H_N(\varepsilon)) \geq 1-\mexp{-(\frac{c''M}{\Lambda}-\tilde{C})N\varepsilon^2}\geq 1-\mexp{-BN\varepsilon^2}.\]
        
        (iii) Next, we check (C3) for $\rho = 2$. Relying on Proposition \ref{le2.1}, we have
        \begin{equation*}
        D_2(P_{\theta_0}\Vert P_{\theta}) \leq e^{2(\norm{\mathcal{G}(\theta_0)}_{+\infty}^2+\norm{\mathcal{G}(\theta)}_{+\infty}^2)} \norm{\mathcal{G}(\theta_0)-\mathcal{G}(\theta)}_{L^2_{\lambda}(\mathcal{X},V)}^2.
        \end{equation*}
        Employing the above inequality and (\ref{bound}), we obtain
        \begin{align*}
            &\Pi_N(\theta : D_2(P^{(N)}_{\theta_0}\Vert P^{(N)}_{\theta})\leq C_3N\varepsilon_N^2)\\
            &\geq \Pi_N(\theta : D_2(P_{\theta_0}\Vert P_{\theta})\leq C_3\varepsilon_N^2, \norm{\theta-\theta_0}_{\mathcal{R}}\leq M')\\
            &\geq \Pi_N(\theta : \exp\{4U^2_{\mathcal{G}}(\bar{M})\}\norm{\mathcal{G}(\theta)-\mathcal{G}(\theta_0)}^2_{L^{2}_{\lambda}}\leq C_3\varepsilon_N^2, \norm{\theta-\theta_0}_{\mathcal{R}}\leq M')\\
            &\geq \Pi_N(\theta : \norm{\mathcal{G}(\theta)-\mathcal{G}(\theta_0)}_{L^{2}_{\lambda}}\leq \exp\{-2U_{\mathcal{G}}^2(\bar{M})\}\sqrt{C_3}\varepsilon_N, \norm{\theta-\theta_0}_{\mathcal{R}}\leq M')
        \end{align*}
        for a choosing constant $C_3>0$, some constant $M'>0$ and $\bar{M}=M' + \Vert \theta_0\Vert _{\mathcal{R}}$.
        Then, using Lipschitz condition of $\mathcal{G}$ and Lemma \ref{lem:decenter}, we have
        \begin{align*}
            &\Pi_N(\theta : D_2(P^{(N)}_{\theta_0}\Vert P^{(N)}_{\theta})\leq C_3N\varepsilon_N^2)\\
            &\geq \Pi_N(\theta :\norm{\theta-\theta_0}_{(H^{\kappa})^*} \leq \exp\{-2U_{\mathcal{G}}^2(\bar{M})\}\sqrt{C_3}\varepsilon_N/L_{\mathcal{G}}(\bar{M}), \Vert \theta-\theta_0\Vert _{\mathcal{R}}\leq M')\\
            &\geq e^{-\frac{1}{p}N\varepsilon_N^2\Vert \theta_0\Vert _{B_{pp}^{\alpha}}^p}\cdot\Pi'(\theta :\norm{\theta}_{(H^{\kappa})^*}\leq C_{\mathcal{G}}({\bar{M}})\varepsilon_N(N\varepsilon_N^2)^{\frac{1}{p}}, \Vert \theta\Vert _{\mathcal{R}'}\leq M'(N\varepsilon_N^2)^{\frac{1}{p}})\\
            &\geq e^{-\frac{1}{p}N\varepsilon_N^2\Vert \theta_0\Vert _{B_{pp}^{\alpha}}^p} \Big(\Pi'(\norm{\theta}_{(H^{\kappa})^*}\leq C_{\mathcal{G}}({\bar{M}})\varepsilon_N(N\varepsilon_N^2)^{\frac{1}{p}}) - \Pi'(\Vert \theta\Vert _{\mathcal{R}'} > M'(N\varepsilon_N^2)^{\frac{1}{p}})\Big)
        \end{align*}
        for $C_{\mathcal{G}}({\bar{M}}) = \exp\{-2U_{\mathcal{G}}^2(\bar{M})\}\sqrt{C_3}/L_{\mathcal{G}}(\bar{M})$.
        From the preceding proof in (ii), we have 
        \[
            \Pi'(\Vert \theta\Vert _{\mathcal{R}'} > M'(N\varepsilon_N^2)^{\frac{1}{p}}) \leq c_1\exp(-c_2{M'}^pN\varepsilon_N^2)\leq \exp(-c_2{M'}^p N\varepsilon_N^2/2),
        \]
        and
        \[
            \Pi'(\norm{\theta}_{(H^{\kappa})^*}\leq C_{\mathcal{G}}({\bar{M}})\varepsilon_N(N\varepsilon_N^2)^{\frac{1}{p}}) \geq \exp(-\tilde{C}_{\mathcal{G}}({\bar{M}})N\varepsilon_N^2),
        \]
        for $\tilde{C}_{\mathcal{G}}({\bar{M}})=c_0({C_{\mathcal{G}}({\bar{M}})})^{-\frac{pd}{p(\alpha+{\kappa})-d}}$ and large enough $M'$.
        Using the last three inequalities, for $C_3$ large enough, there exists some constant $C_2>0$ such that
        \[\Pi_N(\theta : D_2(P^{(N)}_{\theta_0}\Vert P^{(N)}_{\theta})\leq C_3N\varepsilon_N^2)\geq\exp(-C_2N\varepsilon_N^2).\]
         Thus, we have proved the condition (C3).
        
        Because the three conditions are verified, Theorem 2.1 in \cite{zhang2020convergence} gives
        \begin{equation*}
      P_{\theta_0}^{(N)}\hat{Q}N\varepsilon_N^{\frac{2\mu}{\mu+1}}\Vert \mathcal{G}(\theta)-\mathcal{G}(\theta_0)\Vert _{L^{2}_{\lambda}}^{\frac{2}{\mu+1}}\lesssim N(\varepsilon_N^2 +\gamma^2_N),
        \end{equation*}
        and
        \[P_{\theta_0}^{(N)}\hat{Q}N\varepsilon_N^{\frac{2\mu+2\nu}{\mu+\nu+1}}[F(\Vert f_{\theta} - f_{\theta_0}\Vert )]^{\frac{2}{\mu+\nu+1}}\lesssim N(\varepsilon_N^{2} +\gamma^2_N).\]
        That is to say, we obtain
        \[P_{\theta_0}^{(N)}\hat{Q}\Vert \mathcal{G}(\theta)-\mathcal{G}(\theta_0)\Vert _{L^{2}_{\lambda}}^{\frac{2}{\mu+1}}\lesssim \varepsilon_N^{\frac{2}{\mu+1}} +\gamma^2_N \cdot \varepsilon_N^{-\frac{2\mu}{\mu+1}},\]
        and
        \[P_{\theta_0}^{(N)}\hat{Q}[F(\Vert f_{\theta} - f_{\theta_0}\Vert )]^{\frac{2}{\mu+\nu+1}}\lesssim \varepsilon_N^{\frac{2}{\mu+\nu+1}} +\gamma^2_N \cdot \varepsilon_N^{-\frac{2\mu+2\nu}{\mu+\nu+1}}.\]
\end{proof}
\begin{proof}[Proof of Theorem \ref{boundgam}]
        The notations $\tilde{\mathcal{Q}}_E, \Psi, {\mathcal{Q}}_E$ are defined identically as in (\ref{OGMFinf}-\ref{GMFinf}). We recall that 
\[R(Q) = \frac{1}{N}\Big(D(Q\Vert \Pi_N) + Q[D(P^{(N)}_{\theta_0}\Vert P^{(N)}_{\theta})]\Big) = \frac{1}{N}D(Q\Vert \Pi_N)+ Q[D(P_{\theta_0}\Vert P_{\theta})].\]
It is necessary to bound $\frac{1}{N}D(Q\Vert \Pi_N)$ and $Q[D(P_{\theta_0}\Vert P_{\theta})]$ respectively. 
We see that the prior $\Pi_N$ can be represented as 
$\Pi_N = \tilde{\Pi}_N \circ \Psi^{-1},$
where
\[\tilde{\Pi}_N = \mathop{\bigotimes}_{l=-1}^{+\infty}\mathop{\bigotimes}_{r\in R_l}Exp(p;0,\sigma_l),\qquad\sigma_l=2^{-l(\alpha+\frac{d}{2}-\frac{d}{p})}(N\varepsilon_N^2)^{-\frac{1}{p}}.\]
We define $\tilde{Q}_N$ to be 
$\mathop{\bigotimes}_{l=-1}^{+\infty}\mathop{\bigotimes}_{r\in R_l}Exp(p;\theta_{0,lr},\tau_l),\quad \theta_{0,lr} = \pdt{\theta_0}{\psi_{lr}}_{L^2(\mathcal{Z})},$
where
\begin{equation*}
    \tau_{l}=\left\{\begin{aligned}
        &2^{-J(\alpha+{\kappa}+d/2)},& l\leq J,\\
        &\sigma_l,&l>J,
    \end{aligned}\right. \qquad 2^{Jd}\simeq N\varepsilon_N^2
\end{equation*}
The probability measure $Q_N$ is defined as the push-forward of $\tilde{Q}_N$ via $\Psi$, that is, $Q_N = \tilde{Q}_N \circ \Psi^{-1}.$
It is easy to see \textcolor{black}{$Q_N\in {\mathcal{Q}}_E$}.
We first consider the upper bound of $\frac{1}{N}D(Q_N\Vert \Pi_N)$. Because the KL divergence decreases under push-forward \cite[section 10]{varadhan1984large}, we have 
    \begin{align*}
        D(Q_N\Vert \Pi_N) = &D(\tilde{Q}_N \circ \Psi^{-1}\Vert \tilde{\Pi}_N \circ \Psi^{-1}) 
        \leq D(\tilde{Q}_N\Vert \tilde{\Pi}_N) \\
        \leq &\sum_{l=-1}^{+\infty}\sum_{r\in R_l} D(Exp(p;\theta_{0,lr},\tau_l)\Vert Exp(p;0,\sigma_l)).
    \end{align*}
Then, it is sufficient to deduce the upper bound of $D(Exp(p;\theta_{0,lr},\tau_l)\Vert Exp(p;0,\sigma_l))$. {\color{black} We note that $C_{p,b}$ denotes the normalization constant of the p-exponential distribution $Exp(p;0,b)$ (see Lemma \ref{lem:moment} for the explicit expression).} When $l\leq J$,
we have
\begin{equation}\label{boundJ}
    \begin{aligned}
        &D(Exp(p;\theta_{0,lr},\tau_l)\Vert Exp(p;0,\sigma_l))\\
        =&\int\frac{1}{C_{p,\tau_l}}\mexp{-\frac{|x-\theta_{0,lr}|^p}{p\tau_l^p}}\Big(-\log C_{p,\tau_l}-\frac{|x-\theta_{0,lr}|^p}{p\tau_l^p}\Big)dx\\
        &-\int\frac{1}{C_{p,\tau_l}}\mexp{-\frac{|x-\theta_{0,lr}|^p}{p\tau_l^p}}\Big(-\log C_{p,\sigma_l}-\frac{|x|^p}{p\sigma_l^p}\Big)dx\\
        \lesssim&\log \frac{C_{p,\sigma_l}}{C_{p,\tau_l}}+\int\frac{1}{C_{p,\tau_l}}\mexp{-\frac{|x|^p}{p\tau_l^p}}\frac{|x|^p+|\theta_{0,lr}|^p}{p\sigma_l^p}dx\\
        \lesssim&\log \frac{C_{p,\sigma_l}}{C_{p,\tau_l}}+\frac{\tau_l^p}{\sigma_l^p}+\frac{|\theta_{0,lr}|^p}{\sigma_l^p}.
    \end{aligned}
\end{equation}
For $l>J$, we have
    \begin{align*}
        &D(Exp(p;\theta_{0,lr},\sigma_l)\Vert Exp(p;0,\sigma_l))\\
        =&\int\frac{1}{C_{p,\sigma_l}}\mexp{-\frac{|x-\theta_{0,lr}|^p}{p\sigma_l^p}}\Big(-\log C_{p,\sigma_l}-\frac{|x-\theta_{0,lr}|^p}{p\sigma_l^p}\Big)dx\\
        &-\int\frac{1}{C_{p,\sigma_l}}\mexp{-\frac{|x-\theta_{0,lr}|^p}{p\tau_l^p}}\Big(-\log C_{p,\sigma_l}-\frac{|x|^p}{p\sigma_l^p}\Big)dx\\
        =&\int\frac{1}{C_{p,\sigma_l}}\mexp{-\frac{|x|^p}{p\sigma_l^p}}\Big(\frac{|x+\theta_{0,lr}|^p-|x|^p}{p\sigma_l^p}\Big)dx\\
        =&\int\frac{1}{C_{p,\sigma_l}}\mexp{-\frac{|x|^p}{p\sigma_l^p}}\Big(\frac{|x+\theta_{0,lr}|^p + |x-\theta_{0,lr}|^p -2|x|^p}{2p\sigma_l^p}\Big)dx.
    \end{align*}
From the proof of Proposition 2.11 in \cite{agapiou2021rates}, we know that for $p\in[1,2]$
\begin{equation*}
    |x+\theta_{0,lr}|^p + |x-\theta_{0,lr}|^p -2|x|^p \leq 2|\theta_{0,lr}|^p.
\end{equation*}
Thus, we have
\begin{equation*}\label{boundinf}
    D(Exp(p;\theta_{0,lr},\sigma_l)\Vert Exp(p;0,\sigma_l)) \leq \int\frac{1}{C_{p,\sigma_l}}\mexp{-\frac{|x|^p}{p\sigma_l^p}}\frac{|\theta_{0,lr}|^p}{p\sigma_l^p}dx = \frac{|\theta_{0,lr}|^p}{p\sigma_l^p}.
\end{equation*}
The last inequality and \eqref{boundJ} together imply that
\begin{equation*}
    D(Q_N\Vert \Pi_N)
    \lesssim \sum_{l=-1}^{J}\sum_{r\in R_l}\Big(\log \frac{C_{p,\sigma_l}}{C_{p,\tau_l}}+\frac{\tau_l^p}{\sigma_l^p}\Big) + \sum_{l=-1}^{+\infty}\sum_{r\in R_l} \frac{|\theta_{0,lr}|^p}{p\sigma_l^p}.
\end{equation*}
By $2^{Jd} \simeq N\varepsilon_N^2$ and $|R_l| \leq c_0 2^{ld}$, we deduce that
\begin{gather*}
    \sum_{l=-1}^J\sum_{r\in R_l}\log \frac{C_{p,\sigma_l}}{C_{p,\tau_l}}\lesssim 2^{Jd}\log \frac{\sigma_l}{\tau_l} \lesssim N\varepsilon_N^2 \log N,\quad 
    \sum_{l=-1}^J\sum_{r\in R_l}\frac{\tau_l^p}{\sigma_l^p}\lesssim 2^{Jd}\lesssim N\varepsilon_N^2.
\end{gather*}
For the last term, by the wavelet characterization of Sobolev norms (see \cite[section 4]{gin2015mathematical}), we deduce that
\begin{equation*}
    \sum_{l=-1}^{+\infty}\sum_{r\in R_l}\frac{|\theta_{0,lr}|^p}{\sigma_l^p} = \sum_{l=-1}^{+\infty}\sum_{r\in R_l} N\varepsilon_N^22^{pl(\alpha+\frac{d}{2}-\frac{d}{p})}|\theta_{0,lr}|^p = \norm{\theta_0}_{B^{\alpha}_{pp}(\mathcal{Z})}^pN\varepsilon_N^2\lesssim N\varepsilon_N^2,
\end{equation*}
which indicates that 
\begin{equation*}
    \frac{1}{N}D(Q_N\Vert \Pi_N) \leq \frac{1}{N}\sum_{l=-1}^{+\infty}\sum_{r\in R_l} D(Exp(p;\theta_{0,lr},\tau_l)\Vert Exp(p;0,\sigma_l))\lesssim \varepsilon_N^2\log N.
\end{equation*}

Next, we give an upper bound of $Q_N[D(P_{\theta_0}\Vert P_{\theta})]$. We assume independent random variables $\theta_{lr} = \theta_{0,lr} + \tau_l Z_{lr}$ where $Z_{lr} \sim Exp(p;0,1)$ for $r \in R_l, l \in \{-1,0,\dots,{+\infty}\}$. Using Proposition \ref{le2.1} and the condition (\ref{lip}), we have 
    \begin{align*}
        Q_ND(P_{\theta_0}\Vert P_{\theta}) = &\frac{1}{2}Q_N\norm{\mathcal{G}(\theta)-\mathcal{G}(\theta_0)}^2_{L^2_{\lambda}(\mathcal{X},V)} \\
        \lesssim & Q_N[(1+\norm{\theta}_{B_{pp}^{\alpha'}(\mathcal{Z})}^{2l})\norm{\theta-\theta_0}^2_{(H^{{\kappa}}(\mathcal{Z}))^*}]
        \\
        \lesssim & Q_N[(1+\norm{\theta}_{B_{\gamma\gamma}^{\beta'}(\mathcal{Z})}^{2l})\norm{\theta-\theta_0}^2_{(H^{{\kappa}}(\mathcal{Z}))^*}],
    \end{align*}     
for $\beta'<\alpha-d/p$ and large enough $\gamma\geq 2l$ such that $\beta'>\alpha'+d/\gamma$.
Because $Q_N$ is defined by the law of
$\sum_{l=-1}^{+\infty}\sum_{r\in R_l}\theta_{lr}\chi\psi_{lr}, \quad \theta_{lr}\sim Exp(p;\theta_{0,lr},\tau_l),$
we further have
\begin{equation}\label{thm3.5_1}
    \begin{aligned}
        Q_ND(P_{\theta_0}\Vert P_{\theta})\lesssim &E\big[\norm{\theta_0+\sum_{l=-1}^{+\infty}\sum_{r\in R_l}\tau_lZ_{lr}\chi\psi_{lr}}^{2l}_{B_{\gamma\gamma}^{\beta'}(\mathcal{Z})} \\ & \cdot\norm{\sum_{l=-1}^{+\infty}\sum_{r\in R_l}(\theta_{lr}-\theta_{0,lr})\chi\psi_{lr}}^2_{(H^{{\kappa}}(\mathcal{Z}))^*}\big]\\
        \lesssim &E \big[\big(\norm{\theta_0}^{2l}_{B_{\gamma\gamma}^{\beta'}(\mathcal{Z})}+\norm{\sum_{l=-1}^{+\infty}\sum_{r\in R_l}\tau_lZ_{lr}\chi\psi_{lr}}^{2l}_{B_{\gamma\gamma}^{\beta'}(\mathcal{Z})}\big) \\ &
        \cdot\norm{\sum_{l=-1}^{+\infty}\sum_{r\in R_l}(\theta_{lr}-\theta_{0,lr})\chi\psi_{lr}}^2_{(H^{{\kappa}}(\mathcal{Z}))^*}\big],
    \end{aligned}   
\end{equation}
where $\theta_{lr} = 0$ for all $l>J$.
Using \eqref{Sobolevinter2} and wavelet characterization of Sobolev norms (see \cite[section 4]{gin2015mathematical}), we deduce that
    \begin{align*}
        \norm{\sum_{l=-1}^{+\infty}\sum_{r\in R_l}\tau_lZ_{lr}\chi\psi_{lr}}_{B_{\gamma\gamma}^{\beta'}(\mathcal{Z})}^{2l} =& \norm{\chi\sum_{l=-1}^{+\infty}\sum_{r\in R_l}\tau_lZ_{lr}\psi_{lr}}_{B_{\gamma\gamma}^{\beta'}(\mathbb{R}^d)}^{2l}\\
        \lesssim & \norm{\sum_{l=-1}^{+\infty}\sum_{r\in R_l}\tau_lZ_{lr}\psi_{lr}}_{B_{\gamma\gamma}^{\beta'}(\mathbb{R}^d)}^{2l} \\
        \lesssim & \Big(\sum_{l=-1}^{+\infty}\sum_{r\in R_l}2^{\gamma l(\beta'+\frac{d}{2}-\frac{d}{\gamma})}\tau_l^{\gamma}|Z_{lr}|^{\gamma}\Big)^{\frac{2l}{\gamma}}\\
        \lesssim & 1 + \sum_{l=-1}^{+\infty}\sum_{r\in R_l}2^{\gamma l(\beta'+\frac{d}{2}-\frac{d}{\gamma})}\tau_l^{\gamma}|Z_{lr}|^{\gamma},
    \end{align*}
and
     \begin{align*}
     \norm{\sum_{l=-1}^{+\infty}\sum_{r\in R_l}(\theta_{lr}-\theta_{0,lr})\chi\psi_{lr}}_{(H^{{\kappa}}(\mathcal{Z}))^*} = &\sup_{\norm{\phi}_{H^{\kappa}(\mathcal{Z})}\leq 1}\pdt{\sum_{l=-1}^{+\infty}\sum_{r\in R_l}(\theta_{lr}-\theta_{0,lr})\psi_{lr}}{\chi\phi}_{L^2(\mathcal{Z})} \\
     = &\sup_{\norm{\phi}_{H^{\kappa}(\mathcal{Z})}\leq 1}\pdt{\sum_{l=-1}^{+\infty}\sum_{r\in R_l}(\theta_{lr}-\theta_{0,lr})\psi_{lr}}{\chi\phi}_{L^2(\mathbb{R}^d)} \\
     \lesssim &\sup_{\norm{\phi}_{H_c^{\kappa}(\mathcal{Z})}\leq 1}\pdt{\sum_{l=-1}^{+\infty}\sum_{r\in R_l}(\theta_{lr}-\theta_{0,lr})\psi_{lr}}{\phi}_{L^2(\mathbb{R}^d)}\\
     \lesssim & \norm{\sum_{l=-1}^{+\infty}\sum_{r\in R_l}(\theta_{lr}-\theta_{0,lr})\psi_{lr}}_{H^{{-\kappa}}(\mathbb{R}^d)} \\
     \lesssim & \sqrt{\sum_{l=-1}^{+\infty}\sum_{r\in R_l}2^{-2l{\kappa}}(\theta_{lr}-\theta_{0,lr})^2}.
     \end{align*}
     Applying the last two inequalities to \eqref{thm3.5_1}, we obtain
    \begin{align*}
        Q_ND(P_{\theta_0}\Vert P_{\theta})\lesssim& E\big(1 + \sum_{l=-1}^{+\infty}\sum_{r\in R_l}2^{\gamma l(\beta'+\frac{d}{2}-\frac{d}{\gamma})}\tau_l^{\gamma}|Z_{lr}|^{\gamma}\big)\cdot\sum_{l=-1}^{+\infty}\sum_{r\in R_l}2^{-2l{\kappa}}(\theta_{lr}-\theta_{0,lr})^2 \\
        \lesssim& E\big(1 + \sum_{l=-1}^{+\infty}\sum_{r\in R_l}2^{\gamma l(\beta'+\frac{d}{2}-\frac{d}{\gamma})}\tau_l^{\gamma}|Z_{lr}|^{\gamma}\big)\cdot\sum_{l=-1}^{+\infty}\sum_{r\in R_l}2^{-2l{\kappa}} \tau_l^2 Z_{lr}^2\\
        \lesssim& E\big(1 + \sum_{l=-1}^J\sum_{r\in R_l}2^{\gamma l(\beta'+\frac{d}{2}-\frac{d}{\gamma})}\tau_l^{\gamma}|Z_{lr}|^{\gamma}\big)\cdot\sum_{l=-1}^J\sum_{r\in R_l}\tau_l^2 Z_{lr}^2\\ &+  E\big(1 + \sum_{l=-1}^J\sum_{r\in R_l}2^{\gamma l(\beta'+\frac{d}{2}-\frac{d}{\gamma})}\tau_l^{\gamma}|Z_{lr}|^{\gamma}\big)\cdot\sum_{l=J+1}^{+\infty}\sum_{r\in R_l}2^{-2l{\kappa}}\sigma_l^2 Z_{lr}^2.
    \end{align*}
Provided that
\begin{gather}\label{Eb1_inf}
    E\sum_{l=-1}^{+\infty}\sum_{r\in R_l}2^{\gamma l(\beta'+\frac{d}{2}-\frac{d}{\gamma})}{\tau_l}^{\gamma}|Z_{lr}|^{\gamma} = O(1),\\ \label{Eb2_inf}  EZ_{\tilde{l}\tilde{r}}^2\sum_{l=-1}^{+\infty}\sum_{r\in R_l}2^{\gamma l(\beta'+\frac{d}{2}-\frac{d}{\gamma})}{\tau_l}^{\gamma}|Z_{lr}|^{\gamma} = O(1),\quad \mbox{for} \ \tilde{r}\in R_{\tilde{l}},\ \tilde{l} \in \{-1,0,\dots,+\infty\},
\end{gather}
then, since $2^{-J({\kappa}+\alpha)} \simeq \varepsilon_N$ and $|R_l| \leq c_0 2^{ld}$, we have
    \begin{align*}
        Q_ND(P_{\theta_0}\Vert P_{\theta})&\lesssim \left(\sum_{l=-1}^J\sum_{r\in R_l}\tau_l^2 + \sum_{l=J+1}^{+\infty}\sum_{r\in R_l}2^{-2l{\kappa}}\sigma_l^2 \right)\\
        &\lesssim 2^{-2J(\alpha+\kappa+\frac{d}{2})}\cdot2^{Jd} + 2^{-2J({\kappa}+\alpha-\frac{d}{p})}\cdot(N\varepsilon_N^2)^{-\frac{2}{p}}\\
        &\lesssim 2^{-2J(\alpha+\kappa)} + 2^{-2J({\kappa}+\alpha-\frac{d}{p})}\cdot2^{-2J\cdot\frac{d}{p}}\lesssim \varepsilon_N^2.
    \end{align*}
\par
Next, we verify (\ref{Eb1_inf}) and (\ref{Eb2_inf}) to complete the proof. For (\ref{Eb1_inf}),
    \begin{align*}
        E\sum_{l=-1}^{+\infty}\sum_{r\in R_l}2^{\gamma l(\beta'+\frac{d}{2}-\frac{d}{\gamma})}{\tau}_l^{\gamma}|Z_{lr}|^{\gamma} &\lesssim \sum_{l=-1}^{+\infty}\sum_{r\in R_l}2^{\gamma l(\beta' - \alpha +\frac{d}{p}-\frac{d}{\gamma})}E|Z_{lr}|^{\gamma}\\
        & \lesssim \sum_{l=-1}^{+\infty}2^{\gamma l(\beta' - \alpha +\frac{d}{p})}\lesssim 2^{\gamma(\alpha-\beta' -\frac{d}{p})}.
    \end{align*}
For (\ref{Eb2_inf}),  
we split this expectation into to two parts:
\begin{equation*}
    E2^{\gamma \tilde{l}(\beta'+\frac{d}{2}-\frac{d}{\gamma})}{\tau}_{\tilde{l}}^{\gamma}|Z_{\tilde{l}\tilde{r}}|^{\gamma}Z_{\tilde{l}\tilde{r}}^2\cdot E\sum_{(l,r)\neq (\tilde{l},\tilde{r})}2^{\gamma l(\beta'+\frac{d}{2}-\frac{d}{\gamma})}{\tau}_l^{\gamma}|Z_{lr}|^{\gamma}.
\end{equation*}
The second part can be bounded by a constant, as indicated by equation (\ref{Eb1_inf}). For the first part, we have
\begin{equation*}
        E2^{\gamma \tilde{l}(\beta'+\frac{d}{2}-\frac{d}{\gamma})}{\tau}_l^{\gamma}|Z_{\tilde{l}\tilde{r}}|^{\gamma}Z_{\tilde{l}\tilde{r}}^2 \lesssim 2^{\gamma \tilde{l}(\beta' - \alpha +\frac{d}{p}-\frac{d}{\gamma})}E|Z_{\tilde{l}\tilde{r}}|^{\gamma+2}\lesssim 2^{-\gamma(\beta' - \alpha +\frac{d}{p}-\frac{d}{\gamma})}. 
\end{equation*}
Therefore, formula (\ref{Eb2_inf}) is verified.
\end{proof}
\subsection{Contraction rates for high-dimensional Besov priors}
For the consistency theorem with the Gaussian sieve prior, we need the contraction of 
\[\norm{\theta_0 - P_J(\theta_0)}_{(H^{{\kappa}}(\mathcal{Z}))^*},\] 
which is considered in the following lemma.
\begin{lemma}\label{0con}
Assume that $\theta_0$ is supported in the compact subset $K\subset\mathcal{Z}$. Assume that $\norm{\theta_0}_{B_{pp}^{\alpha_0}(\mathcal{Z})}\leq B$ for $\alpha_0=\alpha\vee(\alpha+\frac{d}{p}-\frac{d}{2})$ and some constant $B > 0$. Then,
\begin{equation}
    \norm{\theta_0 - P_J(\theta_0)}_{(H^{{\kappa}}(\mathcal{Z}))^*}\lesssim \varepsilon_N
\end{equation}
for $\varepsilon_N=N^{-\frac{\alpha+{\kappa}}{2\alpha+2{\kappa}+d}}, 2^{J} \simeq N^{\frac{1}{2\alpha +2{\kappa} +d}}$.
\end{lemma}
\begin{proof}
We define compact set $K'$ such that $\mathrm{supp}(\chi)\subset K'\subset \mathcal{Z}$.
\begin{align*}
     \norm{\theta_0 - P_J(\theta_0)}_{(H^{{\kappa}}(\mathcal{Z}))^*} = &\norm{\sum_{l=J+1}^{+\infty}\sum_{r\in R_l}\theta_{0,lr}\chi\psi_{lr}}_{(H^{{\kappa}}(\mathcal{Z}))^*} \\ =&\sup_{\norm{\phi}_{H^{\kappa}(\mathcal{Z})}\leq 1}\pdt{\sum_{l=J+1}^{+\infty}\sum_{r\in R_l}\theta_{0,lr}\psi_{lr}}{\chi\phi}_{L^2(\mathcal{Z})} \\
     \lesssim &\sup_{\norm{\phi}_{H_{K'}^{\kappa}(\mathcal{Z})}\leq 1}\pdt{\sum_{l=J+1}^{+\infty}\sum_{r\in R_l}\theta_{0,lr}\psi_{lr}}{\phi}_{L^2(\mathbb{R}^d)}\\
     \lesssim & \norm{\sum_{l=J+1}^{+\infty}\sum_{r\in R_l}\theta_{0,lr}\psi_{lr}}_{H^{{-\kappa}}(\mathbb{R}^d)} \\
     \lesssim & \sqrt{\sum_{l=J+1}^{+\infty}\sum_{r\in R_l}2^{-2l{\kappa}}\theta_{0,lr}^2}
     \lesssim 2^{-J(\alpha+\kappa)}
     \sqrt{\sum_{l=J+1}^{+\infty}\sum_{r\in R_l}2^{2l{\alpha}}\theta_{0,lr}^2}.
     \end{align*}
For $p\in [1,2]$, we have
\begin{align*}
     \sqrt{\sum_{l=J+1}^{+\infty}\sum_{r\in R_l}2^{2l{\alpha}}\theta_{0,lr}^2} \leq &\Big(\sum_{l=J+1}^{+\infty}\sum_{r\in R_l}2^{pl{\alpha}}\theta_{0,lr}^p\Big)^{\frac{1}{p}}\\\leq&\Big(\sum_{l=J+1}^{+\infty}\sum_{r\in R_l}2^{pl({\alpha_0}+\frac{d}{2}-\frac{d}{p})}\theta_{0,lr}^p\Big)^{\frac{1}{p}} \leq \norm{\theta_0}_{B_{pp}^{\alpha_0}(\mathcal{Z})}.
     \end{align*}
When $p> 2$, 
\begin{align*}
     \sqrt{\sum_{l=J+1}^{+\infty}\sum_{r\in R_l}2^{2l{\alpha}}\theta_{0,lr}^2} \leq \norm{\theta_0}_{H^{\alpha}(\mathcal{Z})} \lesssim \norm{\theta_0}_{B_{pp}^{\alpha_0}(\mathcal{Z})}.
     \end{align*}
Further using $2^{J} \simeq N^{\frac{1}{2\alpha +2{\kappa} +d}}$, we deduce that
\begin{equation*}
\norm{\theta_0 - P_J(\theta_0)}_{(H^{\kappa}(\mathcal{Z}))^*}\lesssim 2^{-J(\alpha+\kappa)} \lesssim \varepsilon_N.
\end{equation*}
\end{proof}
\begin{proof}[Proof of Theorem \ref{mainthmsv}]
Following the proof of Theorem \ref{mainthm}, we will now verify conditions (C1) through (C3) as outlined in \cite[Theorem 2.1]{zhang2020convergence}. 

We first verify conditions (C1) and (C2). When $p\in[1,2]$, we define the sets \( H_N(\varepsilon) \) and \( \Theta_N(\varepsilon) \) as stated in Theorem \ref{mainthm}. Since those lemmas used in the proof of Theorem \ref{mainthm} are also proved for the high-dimensional prior \( \Pi'_J \), the proof for conditions (C1) and (C2) in Theorem \ref{mainthm} also holds true for the set \( \Theta_N(\varepsilon) \). For $p >2$, we define the set
\begin{align*}
H_N(\varepsilon) = \Big\{\theta: \theta = \theta_1 + \theta_2, \norm{\theta_1}_{(H^{k}(\mathcal{Z}))^*}\leq \frac{M^l\varepsilon_N}{L_{\mathcal{G}}(Mr_N(\varepsilon))},\Vert \theta_2 \Vert_{H^{\tilde{\alpha}}(\mathcal{Z})}\leq M\frac{\sqrt{N}\varepsilon}{(\sqrt{N}\varepsilon_{N})^{\frac{2}{p}}}\Big\}  
\end{align*}
 with $\tilde{\alpha}=\alpha+\frac{d}{2}-\frac{d}{p}$. We further define $\Theta_N(\varepsilon) = H_N(\varepsilon) \cap B_{\mathcal{R}}(Mr_N(\varepsilon))\cap\mathrm{span}\{\psi_{lr}\}_{l=-1}^{J}$ for $r\in R_l$, $r_N(\varepsilon)=(\varepsilon/\varepsilon_N)^{\frac{2}{p}}$ and $\varepsilon>\varepsilon_N$. We begin with verification of the condition (C1). Following the same steps in the proof of Theorem \ref{mainthm}, we note that it is sufficient to prove 
\[\log N\Big(B_{H^{\tilde{\alpha}}}\big({M\sqrt{N}\varepsilon}/{(\sqrt{N}\varepsilon_{N})^{\frac{2}{p}}}\big),\Vert\cdot \Vert_{H^{-\kappa}(\mathcal{Z})},\frac{\bar{m}\varepsilon}{C_L(Mr_N(\varepsilon))^l}\Big) \leq \mexp{-cN\varepsilon^2}\]
for some constant $c>0$.
Using \eqref{entropy}, for $\bar{m} = C_LM^{l+1}$ and $\varepsilon > \varepsilon_N$, we deduced that
            \begin{align*} 
            &\log N(B_{H^{\tilde{\alpha}}}\big({M\sqrt{N}\varepsilon}/{(\sqrt{N}\varepsilon_{N})^{\frac{2}{p}}}\big),{\color{black}\Vert\cdot \Vert_{H^{-\kappa}(\mathcal{Z})},}\frac{\bar{m}\varepsilon}{C_L(Mr_N(\varepsilon))^l}) \\
            \leq &C_E\left(\frac{MC_L(Mr_N(\varepsilon))^l(\sqrt{N}\varepsilon)^{1-\frac{2}{p}}}{\bar{m}\varepsilon}\cdot\big(\frac{\varepsilon}{\varepsilon_N}\big)^{\frac{2}{p}}\right)^{d/(\tilde{\alpha} +{\kappa})}\\
            \leq &C_E\left(\frac{M^{l+1}C_{L}}{\bar{m}}\right)^{d/(\tilde{\alpha}+{\kappa})} \cdot (\frac{\varepsilon}{\varepsilon_N})^{\frac{2d(l+1)}{p(\tilde{\alpha}+\kappa)}}\cdot \Big(\frac{(\sqrt{N}\varepsilon)^{1-\frac{2}{p}}}{\varepsilon}\Big)^{\color{black}\frac{d}{\tilde{\alpha}+\kappa}}\\
            \leq &C_E\cdot(\frac{\varepsilon}{\varepsilon_N})^{2}\cdot \big(N^{\frac{p-2}{2p}}\varepsilon_N^{-\frac{2}{p}}\big)^{\color{black}\frac{d}{\tilde{\alpha}+\kappa}}\\
            =&C_E\cdot(\frac{\varepsilon}{\varepsilon_N})^{2}\cdot\big((N\varepsilon_N^2)^{\frac{2\alpha+2\kappa+d}{d}\cdot\frac{p-2}{2p}}\cdot (N\varepsilon_N^2)^{\frac{\alpha+\kappa}{d}\cdot\frac{2}{p}}\big)^{\color{black}\frac{d}{\tilde{\alpha}+\kappa}}\\
            =&C_E\cdot(\frac{\varepsilon}{\varepsilon_N})^{2}\cdot N\varepsilon_N^2
            =C_EN\varepsilon^2,
            \end{align*}          
        where we used $\varepsilon/\varepsilon_N>1$, $p(\tilde{\alpha}+\kappa)\geq p({\alpha}+\kappa)\geq d(l+1) $, and $\varepsilon_N^{-d/(\alpha+\kappa)} = N\varepsilon_N^2$.
        Then, we consider the condition (C2). Following the same steps in the proof of Theorem 3.4, we note that it is sufficient to prove that for some small enough $c'>0$, $\Pi_N(H_N(\varepsilon))$ is bounded below by
        \begin{align*}
            &\Pi_J'\bigg(\theta = \theta_1 + \theta_2 +\theta_3: \norm{\theta_1}_{(H^{\kappa}({\mathbb{R}^d}))^*}\leq \frac{cM^l\varepsilon_N(N\varepsilon_N^2)^\frac{1}{p}}{L_{\mathcal{G}}(Mr_N(\varepsilon))}, \norm{\theta_2}_{B_{pp}^{\alpha}({\mathbb{R}^d})}\leq \frac{c'}{2}(MN\varepsilon^2)^\frac{1}{p},\\
            & \qquad \qquad \qquad \qquad \qquad \norm{\theta_3}_{H^{\alpha+\frac{d}{2}-\frac{d}{p}}({\mathbb{R}^d})} \leq \frac{c}{2}\sqrt{M}\sqrt{N}\varepsilon, \theta_i\in \mathrm{span}\{\psi_{lr}\},i=1,2,3\bigg).
        \end{align*}
        By the definition of $\Pi_N$ and $\Pi'_J$, there exists a small enough constant $c>0$ such that
        \begin{align*}
            \Pi_N(H_N(\varepsilon)) \geq &\Pi_J'\bigg(\theta : \chi\theta = \chi\theta_1 + \chi\theta_2, \norm{\chi\theta_1}_{(H^{\kappa}(\mathcal{Z}))^*}\leq \frac{M^l\varepsilon_N(N\varepsilon_N^2)^\frac{1}{p}}{L_{\mathcal{G}}(Mr_N(\varepsilon))},\\
            &\qquad\qquad\qquad\qquad\qquad\qquad\qquad\norm{\chi\theta_2}_{H^{\tilde{\alpha}}(\mathcal{Z})}\leq {M\sqrt{N}\varepsilon}\bigg)\\
            \geq &\Pi_J'\bigg(\theta = \theta_1 + \theta_2: \norm{\theta_1}_{(H^{\kappa}(\mathbb{R}^d))^*}\leq c\frac{M^l\varepsilon_N(N\varepsilon_N^2)^\frac{1}{p}}{L_{\mathcal{G}}(Mr_N(\varepsilon))},\\
            &\quad \norm{\theta_2}_{H^{\tilde{\alpha}}({\mathbb{R}^d})}\leq c{\sqrt{M}\sqrt{N}\varepsilon}, \theta_i\in \mathrm{span}\{\psi_{lr}\},i=1,2\bigg).
        \end{align*}
        Assuming that $\theta\in B_{pp}^{\alpha}({\mathbb{R}^d})\cap \mathrm{span}\{\psi_{lr}\}$ and $\norm{\theta}_{B_{pp}^{\alpha}({\mathbb{R}^d})}\leq \frac{c'}{2}(MN\varepsilon^2)^\frac{1}{p}$, we deduce that
        \begin{align*}
            \norm{\theta}_{H^{\tilde{\alpha}}({\mathbb{R}^d})}^2 = &\sum_{l=-1}^{J}\sum_{r\in R_l}2^{2l\tilde{\alpha}}\theta_{lr}^2
            \lesssim  \Big(\sum_{l=-1}^{J}2^{ld}\Big)^{\frac{p-2}{p}}\cdot\Big(\sum_{l=-1}^{J}\sum_{r\in R_l}2^{pl(\alpha+\frac{d}{2}-\frac{d}{p})}|\theta_{lr}|^p\Big)^{\frac{2}{p}}\\
            \lesssim &2^{Jd\cdot\frac{p-2}{p}}\cdot\norm{\theta}_{B_{pp}^{\alpha}({\mathbb{R}^d})}^2
        \end{align*}
        Thus, for $c'$ small enough, we have
        \begin{align*}
            \norm{\theta}_{H^{\tilde{\alpha}}({\mathbb{R}^d})} \lesssim 2^{Jd\cdot\frac{p-2}{2p}}\cdot\norm{\theta}_{B_{pp}^{\alpha}({\mathbb{R}^d})}
            \lesssim(\sqrt{N}\varepsilon_N)^{1-\frac{2}{p}}\cdot \frac{c'}{2}(\sqrt{M}\sqrt{N}\varepsilon)^{\frac{2}{p}}
            \leq \frac{c}{2}\sqrt{M}\sqrt{N}\varepsilon,
        \end{align*}
        which indicates that
        \begin{align*}
            &\Pi_N(H_N(\varepsilon)) \\ \geq &\Pi_J'\bigg(\theta = \theta_1 + \theta_2: \norm{\theta_1}_{(H^{\kappa}(\mathbb{R}^d))^*}\leq c\frac{M^l\varepsilon_N(N\varepsilon_N^2)^\frac{1}{p}}{L_{\mathcal{G}}(Mr_N(\varepsilon))},\quad \norm{\theta_2}_{H^{\tilde{\alpha}}({\mathbb{R}^d})}\leq c{\sqrt{M}\sqrt{N}\varepsilon},\\ 
&\qquad\qquad\qquad\qquad\qquad\qquad\qquad\qquad\qquad\qquad\qquad\qquad\theta_i\in \mathrm{span}\{\psi_{lr}\},i=1,2\bigg)\\
            \geq&\Pi_J'\bigg(\theta_1 + \theta_2 +\theta_3: \norm{\theta_1}_{(H^{\kappa}({\mathbb{R}^d}))^*}\leq \frac{cM^l\varepsilon_N(N\varepsilon_N^2)^\frac{1}{p}}{L_{\mathcal{G}}(Mr_N(\varepsilon))}, \norm{\theta_2}_{B_{pp}^{\alpha}({\mathbb{R}^d})}\leq \frac{c'}{2}(MN\varepsilon^2)^\frac{1}{p},\\
            & \qquad \qquad \qquad \qquad \quad \norm{\theta_3}_{H^{\alpha+\frac{d}{2}-\frac{d}{p}}({\mathbb{R}^d})} \leq \frac{c}{2}\sqrt{M}\sqrt{N}\varepsilon, \theta_i\in \mathrm{span}\{\psi_{lr}\},i=1,2,3\bigg).
        \end{align*}
        
Finally, the proof for condition (C3) requires a slight modification because \( \theta_0 \) is no longer within $Z$. Using Lemma \ref{0con} and Lipschitz condition (\ref{lip}), we have
\begin{equation*}
    \norm{\mathcal{G}(\theta_0)-\mathcal{G}(P_J(\theta_0))}_{L^{2}_{\lambda}(\mathcal{X},\mathbb{R})} \leq C_0 \varepsilon_N
\end{equation*}
for some constant $C_0$. Thus, 
\begin{align*}
            &\Pi_N(\theta : D_2(P^{(N)}_{\theta_0}\Vert P^{(N)}_{\theta})\leq C_3N\varepsilon_N^2)\\
            &=\Pi_N(\theta : N\log\int_{\mathcal{X}}\exp\{[\mathcal{G}(\theta)(x)-\mathcal{G}(\theta_0)(x)]^2\}d\lambda(x)\leq C_3N\varepsilon_N^2)\\
            &\geq \Pi_N(\theta : \norm{\mathcal{G}(\theta)-\mathcal{G}(\theta_0)}_{L^{2}_{\lambda}}\leq C_3\exp\{-2U_{\mathcal{G}}^2(\bar{M})\}\varepsilon_N, \norm{\theta-\theta_0}_{\mathcal{R}}\leq M')\\
            &\geq \Pi_N(\theta :\norm{\mathcal{G}(\theta)-\mathcal{G}(P_J(\theta_0))}_{L^2_{\lambda}}\leq C'(\bar{M})\varepsilon_N, \norm{\theta-P_J(\theta_0)}_{\mathcal{R}}\leq M'')
\end{align*}
where the lower bound of the last probability see the proof of Theorem \ref{mainthm}.
\end{proof}
\begin{proof}[Proof of Theorem \ref{boundgamfinite}]
The proof of this theorem is similar to that of Theorem \ref{boundgam}, so we will omit same procedure. The notations $\tilde{\mathcal{Q}}_E^J, \Psi_J, {\mathcal{Q}}_E^J$ are defined identically as in (\ref{OGMF}-\ref{GMF}). We recall that 
\[R(Q) = \frac{1}{N}\Big(D(Q\Vert \Pi_N) + Q[D(P^{(N)}_{\theta_0}\Vert P^{(N)}_{\theta})]\Big) = \frac{1}{N}D(Q\Vert \Pi_N)+ Q[D(P_{\theta_0}\Vert P_{\theta})].\]
It is necessary to bound $\frac{1}{N}D(Q\Vert \Pi_N)$ and $Q[D(P_{\theta_0}\Vert P_{\theta})]$ respectively. 
We define $\tilde{Q}_N$ to be 
$\mathop{\bigotimes}_{l=-1}^J\mathop{\bigotimes}_{r\in R_l}Exp(q;\theta_{0,lr},\tau), \quad \tau \simeq 2^{-J(\alpha+{\kappa}+d/2)},\quad \theta_{0,lr} = \pdt{\theta_0}{\psi_{lr}}_{L^2(\mathcal{Z})}.$
The probability measure $Q_N$ is defined as the push-forward of $\tilde{Q}_N$ via $\Psi_J$, that is, $Q_N = \tilde{Q}_N \circ \Psi_J^{-1}.$
It is easy to see $Q_N\in {\mathcal{Q}}_E^J(q)$. We also see that the prior $\Pi_N$ can be represented as 
$\Pi_N = \tilde{\Pi}_N \circ \Psi_J^{-1},$
where $\tilde{\Pi}_N = \mathop{\bigotimes}_{l=-1}^J\mathop{\bigotimes}_{r\in R_l}Exp(p;0,\sigma_l)$ for $\sigma_l=2^{-l(\alpha+\frac{d}{2}-\frac{d}{p})}(N\varepsilon_N^2)^{-\frac{1}{p}}$.
We first consider the upper bound of $\frac{1}{N}D(Q_N\Vert \Pi_N)$. Because the KL divergence decreases under push-forward \cite[section 10]{varadhan1984large}, we have 
\begin{equation*}
    \begin{aligned}
        D(Q_N\Vert \Pi_N) = &D(\tilde{Q}_N \circ \Psi_J^{-1}\Vert \tilde{\Pi}_N \circ \Psi_J^{-1})
        \leq D(\tilde{Q}_N\Vert \tilde{\Pi}_N) \\
        \leq &\sum_{l=-1}^J\sum_{r\in R_l} D(Exp(q;\theta_{0,lr},\tau)\Vert Exp(p;0,\sigma_l)).
    \end{aligned}
\end{equation*}
Then, it is sufficient to consider the upper bound of $D(Exp(q;\theta_{0,lr},b)\Vert Exp(p;0,\sigma_l))$:
\begin{equation*}
        D(Exp(q;\theta_{0,lr},\tau)\Vert Exp(p;0,\sigma_l))
        \lesssim \log \frac{C_{p,\sigma_l}}{C_{q,\tau}}+\frac{\tau^p}{\sigma_l^p}+\frac{|\theta_{0,lr}|^p}{\sigma_l^p}.
\end{equation*}
By $2^{Jd} \simeq N\varepsilon_N^2$ and $d_J \leq c_0 2^{Jd}$, we deduce that
\begin{equation*}
    \sum_{l=-1}^J\sum_{r\in R_l}\log \frac{C_{p,\sigma_l}}{C_{q,b}}\lesssim 2^{Jd}\log \frac{\sigma}{\tau} \lesssim N\varepsilon_N^2 \log N,
\end{equation*}
\begin{equation*}
    \sum_{l=-1}^J\sum_{r\in R_l}\frac{\tau^p}{\sigma_l^p}\lesssim 2^{Jd}\lesssim N\varepsilon_N^2.
\end{equation*}
For the last term, we have
\begin{equation*}
    \sum_{l=-1}^J\sum_{r\in R_l}\frac{|\theta_{0,lr}|^p}{\sigma_l^p} = \sum_{l=-1}^J\sum_{r\in R_l} N\varepsilon_N^22^{pl(\alpha+\frac{d}{2}-\frac{d}{p})}|\theta_{0,lr}|^p = \norm{\theta_0}_{B^{\alpha}_{pp}(\mathcal{Z})}^pN\varepsilon_N^2\lesssim N\varepsilon_N^2.
\end{equation*}
In conclusion, 
\begin{equation*}
    \frac{1}{N}D(Q_N\Vert \Pi_N) \leq \frac{1}{N}\sum_{l=-1}^J\sum_{r\in R_l} D(Exp(q;\theta_{0,lr},\tau)\Vert Exp(p;0,\sigma_l))\lesssim \varepsilon_N^2\log N.
\end{equation*}

Next, we give an upper bound of $Q_N[D(P_{\theta_0}\Vert P_{\theta})]$. We assume independent random variables $\theta_{lr} = \theta_{0,lr} + \tau Z_{lr}$ where $Z_{lr} \sim Exp(q;0,1)$ for $r \in R_l, l \in \{-1,0,\dots,J\}$. Using Proposition \ref{le2.1} and the condition (\ref{lip}), we have 
    \begin{align*}
        Q_ND(P_{\theta_0}\Vert P_{\theta}) = &\frac{1}{2}Q_N\norm{\mathcal{G}(\theta)-\mathcal{G}(\theta_0)}^2_{L^2_{\lambda}(\mathcal{X},V)} \\
        \lesssim & Q_N[(1+\norm{\theta}_{B_{pp}^{\alpha'}(\mathcal{Z})}^{2l})\norm{\theta-\theta_0}^2_{(H^{{\kappa}}(\mathcal{Z}))^*}]
        \\
        \lesssim & Q_N[(1+\norm{\theta}_{B_{\gamma\gamma}^{\beta'}(\mathcal{Z})}^{2l})\norm{\theta-\theta_0}^2_{(H^{{\kappa}}(\mathcal{Z}))^*}],
    \end{align*}     
for $\beta'<\alpha-d/p$ and large enough $\gamma\geq 2l$ such that $\beta'>\alpha'+d/\gamma$.
Because $Q_N$ is defined by the law of
$\sum_{l=-1}^J\sum_{r\in R_l}\theta_{lr}\chi\psi_{lr}, \quad \theta_{lr}\sim Exp(q;\theta_{0,lr},\tau),$
we have
\begin{equation}\label{thm3.8_1}
    \begin{aligned}
        Q_ND(P_{\theta_0}\Vert P_{\theta})
        \lesssim &E \big[\big(\norm{\theta_0}^{2l}_{B_{\gamma\gamma}^{\beta'}(\mathcal{Z})}+\norm{\sum_{l=-1}^J\sum_{r\in R_l}\tau Z_{lr}\chi\psi_{lr}}^{2l}_{B_{\gamma\gamma}^{\beta'}(\mathcal{Z})}\big) \\ &
        \cdot\norm{\sum_{l=-1}^{+\infty}\sum_{r\in R_l}(\theta_{lr}-\theta_{0,lr})\chi\psi_{lr}}^2_{(H^{{\kappa}}(\mathcal{Z}))^*}\big],
    \end{aligned}   
\end{equation}
where $\theta_{lr} = 0$ for all $l>J$.
Using \eqref{Sobolevinter2} and wavelet characterization of Sobolev norms (see \cite[section 4]{gin2015mathematical}), we deduce that
\begin{equation*}
        \norm{\sum_{l=-1}^J\sum_{r\in R_l}\tau Z_{lr}\chi\psi_{lr}}_{B_{\gamma\gamma}^{\beta'}(\mathcal{Z})}^{2l} 
        \lesssim  1 + \sum_{l=-1}^J\sum_{r\in R_l}2^{\gamma l(\beta'+\frac{d}{2}-\frac{d}{\gamma})}\tau^{\gamma}|Z_{lr}|^{\gamma},
\end{equation*}
\begin{align*}
     \norm{\sum_{l=-1}^{+\infty}\sum_{r\in R_l}(\theta_{lr}-\theta_{0,lr})\chi\psi_{lr}}_{(H^{{\kappa}}(\mathcal{Z}))^*} 
     \lesssim \sqrt{\sum_{l=-1}^{+\infty}\sum_{r\in R_l}2^{-2l{\kappa}}(\theta_{lr}-\theta_{0,lr})^2}.
     \end{align*}
     Applying the last two inequalities to \eqref{thm3.8_1}, we obtain
    \begin{align*}
        Q_ND(P_{\theta_0}\Vert P_{\theta})
        \lesssim& E\big(1 + \sum_{l=-1}^J\sum_{r\in R_l}2^{\gamma l(\beta'+\frac{d}{2}-\frac{d}{\gamma})}\tau^{\gamma}|Z_{lr}|^{\gamma}\big)\cdot\sum_{l=-1}^J\sum_{r\in R_l}\tau^2 Z_{lr}^2\\ &+  E\big(1 + \sum_{l=-1}^J\sum_{r\in R_l}2^{\gamma l(\beta'+\frac{d}{2}-\frac{d}{\gamma})}\tau^{\gamma}|Z_{lr}|^{\gamma}\big)\cdot2^{-2J({\kappa}+\alpha)}.
    \end{align*}
Provided that
\begin{equation}\label{Eb1}
    E\sum_{l=-1}^J\sum_{r\in R_l}2^{\gamma l(\beta'+\frac{d}{2}-\frac{d}{\gamma})}\tau^{\gamma}|Z_{lr}|^{\gamma} = O(1),
\end{equation}
\begin{equation}\label{Eb2}
    EZ_{\tilde{l}\tilde{r}}^2\sum_{l=-1}^J\sum_{r\in R_l}2^{\gamma l(\beta'+\frac{d}{2}-\frac{d}{\gamma})}\tau^{\gamma}|Z_{lr}|^{\gamma} = O(1),\quad \mbox{for} \ \tilde{r}\in R_{\tilde{l}},\ \tilde{l} \in \{-1,0,\dots,J\},
\end{equation}
then, since $2^{-J({\kappa}+\alpha)} \simeq \varepsilon_N$ and $d_J \leq c_0 2^{Jd}$, we have
\begin{equation*}
    \begin{aligned}
        Q_ND(P_{\theta_0}\Vert P_{\theta})&\lesssim E\big(1 + \sum_{l=-1}^J\sum_{r\in R_l}2^{\gamma l(\beta'+\frac{d}{2}-\frac{d}{\gamma})}\tau^{\gamma}|Z_{lr}|^{\gamma}\big)\cdot\left(\sum_{l=-1}^J\sum_{r\in R_l}\tau^2 Z_{lr}^2 + 2^{-2J({\kappa}+\alpha)}\right)\\
        &\lesssim \tau^2d_J + 2^{-2J({\kappa}+\alpha)}\lesssim \varepsilon_N^2.
    \end{aligned}
\end{equation*}

Next, we verify (\ref{Eb1}) and (\ref{Eb2}) to complete the proof. For (\ref{Eb1}),
    \begin{align*}
        E\sum_{l=-1}^J\sum_{r\in R_l}2^{\gamma l(\beta'+\frac{d}{2}-\frac{d}{\gamma})}b^{\gamma}|Z_{lr}|^{\gamma} &\lesssim \sum_{l=-1}^J\sum_{r\in R_l}2^{\gamma l(\beta' - \alpha +\frac{d}{p}-\frac{d}{\gamma})}E|Z_{lr}|^{\gamma}\\
        & \lesssim \sum_{l=-1}^J2^{\gamma l(\beta' - \alpha +\frac{d}{p})}\lesssim 2^{\gamma(\alpha-\beta' -\frac{d}{p})}.
    \end{align*}
For (\ref{Eb2}),  
we split this expectation into to two parts:
\begin{equation*}
    E2^{\gamma \tilde{l}(\beta'+\frac{d}{2}-\frac{d}{\gamma})}\tau^{\gamma}|Z_{\tilde{l}\tilde{r}}|^{\gamma}Z_{\tilde{l}\tilde{r}}^2\cdot E\sum_{(l,r)\neq (\tilde{l},\tilde{r})}2^{\gamma l(\beta'+\frac{d}{2}-\frac{d}{\gamma})}\tau^{\gamma}|Z_{lr}|^{\gamma}.
\end{equation*}
The second part can be bounded by a constant, as indicated by equation (\ref{Eb1}). For the first part, we have
\begin{equation*}
        E2^{\gamma \tilde{l}(\beta'+\frac{d}{2}-\frac{d}{\gamma})}\tau^{\gamma}|Z_{\tilde{l}\tilde{r}}|^{\gamma}Z_{\tilde{l}\tilde{r}}^2 \lesssim 2^{\gamma \tilde{l}(\beta' - \alpha +\frac{d}{p}-\frac{d}{\gamma})}E|Z_{\tilde{l}\tilde{r}}|^{\gamma+2}\lesssim 2^{-\gamma(\beta' - \alpha +\frac{d}{p}-\frac{d}{\gamma})}. 
\end{equation*}
Therefore, formula (\ref{Eb2}) is verified.
\end{proof}
\subsection{Contraction rates of the Darcy flow problem}
In this subsection, we will prove Theorem 4.1 by verifying Conditions 3.1 and 3.2 using results from our work \cite{zu2024consistencyvariationalbayesianinference}. We will provide regularity and conditional stability estimates for the forward map \( \mathcal{G}(\theta) \) as defined in (4.3) in the proof of Theorem 4.1.
\begin{proof}[Proof of Theorem 4.1]
We will verify Conditions 3.1 and 3.2 for the forward map \( \mathcal{G} \) of problem (4.1) with \( \mathcal{R} = B^{b}_{pp} \) for some $b<\alpha-d/p$. From Lemma B.2 in \cite{zu2024consistencyvariationalbayesianinference}, we have
\[\mathop{\sup}_{\theta \in B_{H^t}(M)}\Vert u_{f_{\theta}} \Vert_{H^{t+1}} \leq CM^{t^3+t^2}\]
for $2<t-d/2<\alpha-2d/p$. The above inequality, combined with the Sobolev embedding \( H^2 \subset C^0, H^t\subset B^b_{pp} \) for $b-d/p>t-d/2$, implies that
\[\mathop{\sup}_{\theta \in B_{\mathcal{R}}(M)}\mathop{\sup}_{x \in \mathcal{X}}\vert\mathcal{G}(\theta)(x)\vert \leq CM^{\mu}.\]
Therefore, we have condition (3.1) verified with $\mu = t^3 + t^2$.
For $\theta_1,\theta_2 \in H^{s}$ with $2<s-d/2<\alpha-2d/p$, Lemma B.3 in \cite{zu2024consistencyvariationalbayesianinference} implies
\[\Vert \mathcal{G}(\theta_1)-\mathcal{G}(\theta_2)\Vert _{L^2} \leq C (1+\norm{\theta_1}_{C^1}^3\vee\norm{\theta_2}_{C^1}^3)\Vert \theta_1-\theta_2\Vert _{(H^1)^*}.\]
Then, by Sobolev embedding $H^s \subset B^b_{pp}$ with $b-d/p>s-d/2$,
\[\Vert \mathcal{G}(\theta_1)-\mathcal{G}(\theta_2)\Vert _{L^2} \leq C (1+\norm{\theta_1}_{\mathcal{R}}^3\vee\norm{\theta_2}_{\mathcal{R}}^3)\Vert \theta_1-\theta_2\Vert _{(H^1)^*}\]
which verifies condition (3.2) with ${\kappa}=1$ and $l=3$.
For condition (3.3), Lemma B.4 in \cite{zu2024consistencyvariationalbayesianinference} implies that
\[\Vert f_{\theta} - f_{\theta_0}\Vert _{L^2}^{\frac{s+1}{s-1}} \leq C M^{\frac{(2s^2+1)(s+1)}{s-1}}\Vert u_f-u_{f_0}\Vert _{L^2}\]
for $\theta \in B_{\mathcal{R}}(M)$.

Given our requirement on \( \alpha \), we have
\[\alpha + {\kappa} \geq 4d/p = {d(l+1)}/{p}.
\]
In conclusion, using Theorem \ref{finalthm} with the conditions verified above, we have proved Theorem \ref{mainthmDarcy}
\end{proof}
\subsection{Contraction rates of the Inverse potential problem for a subdiffusion equation}
Using the regularity and stability estimates presented in our work \cite{zu2024consistencyvariationalbayesianinference}, we now aim to prove Theorem 4.3.
\begin{proof}[Proof of Theorem 4.3]
We will verify the Conditions 3.1 and 3.2 for the forward map $\mathcal{G}$ as defined in (4.9) 
with $\mathcal{R} = B^{b}_{pp}$ for some $b<\alpha-d/p$.
From Lemma B.8 in \cite{zu2024consistencyvariationalbayesianinference}, for $b>d/p$, we have
\[\mathop{\sup}_{\theta \in B_{\mathcal{R}}(M)}\Vert u_{q_{\theta}}(T) \Vert_{H^{2}} \leq c(1+T^{-\beta}).\]
This inequality, combined with the Sobolev embedding $H^2 \subset C$, imply that
\[\mathop{\sup}_{\theta \in B_{\mathcal{R}}(M)}\mathop{\sup}_{x \in \Omega}\vert\mathcal{G}(\theta)(x)\vert \leq c(1+T^{-\beta}).\]
Therefore, we have verified condition (3.1) with $\mu = 0$.
Lemma B.7 in \cite{zu2024consistencyvariationalbayesianinference}, in conjunction with Lemma 29 in \cite{IntroNonLinear_nickl2020convergence} and Sobolev embedding theorem, implies that
\[\lVert \mathcal{G}(\theta_1)-\mathcal{G}(\theta_2)\rVert _{L^2} \leq c(1+T^{-\beta})(1+\norm{\theta_1}_{\mathcal{R}}^2\vee\norm{\theta_2}_{\mathcal{R}}^2)\lVert \theta_1-\theta_2\rVert _{(H_0^2(\Omega))^*}\]
for $\theta_1,\theta_2 \in \mathcal{R}$ with $b>2+d/p$, which verifies condition (3.2) with ${\kappa}=2$, $l=2$.
For condition (3.3), Lemma B.9 in \cite{zu2024consistencyvariationalbayesianinference}, combined with Lemma 29 in \cite{IntroNonLinear_nickl2020convergence}, implies that, for any integer $s$ such that $0<s<\alpha+d/2-2d/p$ and $b-d/p>s-d/2$,
\[\Vert q_{\theta} - q_{0}\Vert_{L^2}^{\frac{2+s}{s}} \leq C M^{2+4s}\Vert \mathcal{G}(\theta)-\mathcal{G}(\theta_0)\Vert_{L^2},\quad \theta \in B_{\mathcal{R}}(M).\]
Given our requirement on \( \alpha \), we have
$\alpha + {\kappa} \geq {3} \geq {d(l+1)}/{p}.$
In conclusion, using Theorem \ref{finalthm} with conditions verified above, we have proved Theorem \ref{mainthmFrac}.
\end{proof}
Theorem \ref{Fracminmax} is a direct corollary of Theorem 4.4 in our work \cite{zu2024consistencyvariationalbayesianinference}.
\section{Proofs of the concentration of Besov priors}\label{sec:proofBesov}
\subsection{Some facts about the univariate p-exponential distribution}
\begin{lemma}\label{lem:moment}
    Let $\xi \sim Exp(p;0,b)$, $p\geq 1$ and $b>0$. Then, for any $t\geq0$, we have
    \[ E_{\xi}\abs{\xi}^{t} = p^{\frac{t}{p}}b^{t}\Gamma(\frac{t+1}{p})/\Gamma(\frac{1}{p}).\]
    Particularly, when $t=p$, notice that
    \[\mathbb{E}\abs{\xi}^{p} = b^p.\]
\end{lemma}
\begin{proof}
    We first calculate the normalization constant $C_{p,b}$:
    \[C_{p,b} =  \int_{\mathbb{R}}\mexp{-\frac{\abs{x}^p}{pb^p}}dx = 2\int_{0}^{+\infty}\mexp{-\frac{x^p}{pb^p}}dx.\]
    Let $z:=\frac{x^p}{pb^p}$ and we deduce that
    \[C_{p,b} = 2p^{\frac{1}{p}-1}b\int_{0}^{+\infty} z^{\frac{1}{p}-1}e^{-z}dz = 2p^{\frac{1}{p}-1}\Gamma(\frac{1}{p})b,\]
    where $\Gamma(\cdot)$ is the Gamma function.
    Then, for $E_{\xi}\abs{\xi}^{t}$, we have
    \[C_{p,b}\cdot\mathbb{E}\abs{\xi}^{t} = \int_{\mathbb{R}}\abs{x}^t\mexp{-\frac{\abs{x}^p}{pb^p}}dx = 2\int_{0}^{+\infty}x^t\mexp{-\frac{x^p}{pb^p}}dx.\]   
    Let $z:=\frac{x^p}{pb^p}$ and we deduce that
    \[C_{p,b}\cdot\mathbb{E}\abs{\xi}^{t} = 2p^{\frac{t+1}{p}-1}b^{t+1}\int_{0}^{+\infty} z^{\frac{t+1}{p}-1}e^{-z}dx = 2p^{\frac{t+1}{p}-1}\Gamma(\frac{t+1}{p})b^{t+1}.\]
    Thus, we finally have
    \[\mathbb{E}\abs{\xi}^{t} = 2p^{\frac{t+1}{p}-1}\Gamma(\frac{t+1}{p})b^{t+1}/C_{p,b} = p^{\frac{t}{p}}b^{t}\Gamma(\frac{t+1}{p})/\Gamma(\frac{1}{p}).\]
    Particularly, when $t=p$, notice that
    \[\mathbb{E}\abs{\xi}^{p} =  pb^{p}\Gamma(1+\frac{1}{p})/\Gamma(\frac{1}{p}) = b^p,\]
    where we used the property of Gamma function that $\Gamma(1+\frac{1}{p}) = \frac{1}{p}\cdot\Gamma(\frac{1}{p})$.
\end{proof}
\subsection{Concentration of Besov priors}
\begin{lemma}\label{lem:concentration}
    Consider $\Pi'$ defined in Definition \ref{Def:BesovPrior}. Assume $\alpha > b+ \frac{d}{p}$ for some $b\geq0$. Then, there exist constants $c_1,c_2>0$ such that for any $r>0$
    \[\Pi'(\norm{\theta}_{B^b_{pp}(\mathbb{R}^d)}\geq r) \leq c_1e^{-c_2r^p}.\]
\end{lemma}
\begin{proof}
    Provided that there exists a constant $c_2$ such that 
    \begin{equation} \label{lem:concentration:1}
        E_{\Pi'}[\exp(c_2\norm{\theta}^p_{B^b_{pp}(\mathbb{R}^d)})] < \infty,
    \end{equation}
    Markov inequality implies that
    \begin{equation*}
        \Pi'(\norm{\theta}_{B^b_{pp}(\mathbb{R}^d)}\geq r) \leq E_{\Pi'}[\exp(c_2\norm{\theta}^p_{B^b_{pp}(\mathbb{R}^d)})]e^{-c_2r^p}.
    \end{equation*}
    Then, we are going to prove \eqref{lem:concentration:1} to complete the proof. By the definition of $\Pi'$ and wavelet characterization of Sobolev norms, we deduce that 
    \begin{align*}
       E_{\Pi'}[\exp(c_2\norm{\theta}^p_{B^b_{pp}(\mathbb{R}^d)})] =  E\mexp{c_2\sum_{l=-1}^{+\infty}\sum_{r\in R_l}2^{-pl(\alpha-b)}\abs{\xi_{lr}}^p}.
    \end{align*}
    Combining with the property of p-exponential distributions, let $c_2p2^{p(\alpha-b)}\leq \frac{1}{2}$ and we have
    \begin{align*}
       E_{\Pi'}[\exp(c_2\norm{\theta}^p_{B^b_{pp}(\mathbb{R}^d)})] =  &\prod_{l=-1}^{+\infty}\prod_{r\in R_l} E\mexp{c_22^{-pl(\alpha-b)}\abs{\xi_{lr}}^p}\\
       =&\prod_{l=-1}^{+\infty}\prod_{r\in R_l}\big(1-c_2p2^{-pl(\alpha-b)}\big)^{-\frac{1}{p}} \\
       =&\prod_{l=-1}^{+\infty}\big(1-c_2p2^{-pl(\alpha-b)}\big)^{-\frac{c_02^{ld}}{p}}.
    \end{align*}
    Because 
    $(1-c_2p2^{-pl(\alpha-b)}\big)^{-\frac{2^{pl(\alpha-b)}}{c_2p}}\rightarrow e$ when $l\rightarrow +\infty,$
    which indicates that
    \[(1-c_2p2^{-pl(\alpha-b)}\big)^{-\frac{2^{pl(\alpha-b)}}{c_2p}}< C,\quad \forall \, l\geq -1,\]for some constant $C>0$, we further have
    \begin{align*}
       E_{\Pi'}[\exp(c_2\norm{\theta}^p_{B^b_{pp}(\mathbb{R}^d)})] \leq  &\prod_{l=-1}^{+\infty}C^{{c_0c_22^{l(d-p(\alpha-b))}}}
       = C^{\sum_{l=-1}^{+\infty}{c_0c_22^{l(d-p(\alpha-b))}}}\\
       \lesssim& C^{c_0c_22^{p(\alpha-b)-d}},
    \end{align*}
    where we also used $\alpha >b+\frac{d}{p}$.
\end{proof}
Next lemma improve the regularity of functions above using truncated prior $\Pi'_J$. 
\begin{lemma}\label{lem:concentrationJ}
    Consider $\Pi'_J$ defined in Definition \ref{Def:TrunBesovPrior}. Then, there exist constants $c$ such that for any $r>0$
    \[\Pi'_J(\norm{\theta}_{B^{\alpha}_{pp}(\mathbb{R}^d)}\geq r) \leq \exp\{-c(r^p-2^{Jd})\}.\]
\end{lemma}
\begin{proof}
    Combining with the property of p-exponential distributions, let $1-c'p\geq \frac{1}{e}$ and we have
    \begin{align*}
       E_{\Pi'_J}[\exp(c'\norm{\theta}^p_{B^{\alpha}_{pp}(\mathbb{R}^d)})] =  &\prod_{l=-1}^{J}\prod_{r\in R_l} E\mexp{c'\abs{\xi_{lr}}^p}
       =\prod_{l=-1}^{J}\prod_{r\in R_l}\big(1-c'p\big)^{-\frac{1}{p}} \\
       \leq&\prod_{l=-1}^{J}\mexp{\frac{c_02^{ld}}{p}}
       \leq\mexp{c''2^{Jd}}
    \end{align*}
    Then, the Markov inequality implies that
    \begin{align*}
        \Pi'_J(\norm{\theta}_{B^{\alpha}_{pp}(\mathbb{R}^d)}\geq r) &\leq E_{\Pi'_J}[\exp(c'\norm{\theta}^p_{B^{\alpha}_{pp}(\mathbb{R}^d)})]e^{-c'r^p}\\
        &\leq \mexp{-c'r^p+c''2^{Jd}} \leq \mexp{-c(r^p-2^{Jd})}.
    \end{align*}
\end{proof}
\par
The proof of the next lemma on small ball probability of $\Pi'$ is similar to that of Lemma 6.3 in \cite{agapiou2024laplace}.
\begin{lemma}\label{lem:smallball}
    Consider $\Pi'$ defined in Definition \ref{Def:BesovPrior}. Let $\kappa\geq 0$. Then, there exist a constant $C>0$ such that for any $0<r<1$,
    \[-\log\Pi'(\norm{\theta}_{(H^{\kappa}(\mathbb{R}^d))^*}\leq r) \leq Cr^{-\frac{pd}{p(\alpha+\kappa)-d}}.\]
\end{lemma}
\begin{proof}
     By the definition of $\Pi'$ and display (B.3) of the Supplement \cite{agapiou2024laplaceSupp}, for some $c>0$ we have
     \[\Pi'(\norm{\theta}_{(H^{\kappa}(\mathbb{R}^d))^*}\leq r) \geq \mathbb{P}\Big(\sum_{l=-1}^{+\infty}(l^{\frac{1}{d}(-\alpha-\kappa-\frac{d}{2}+\frac{d}{p})}\xi_l)^2\leq cr^2\Big),\quad \xi_l \sim Exp(p;0,1).\]
     Then, Theorem 4.2 in \cite{aurzada2007lower} (with $\mu=\frac{\alpha+\kappa}{d}+\frac{1}{2}-\frac{1}{p}$ and $p=2$ which is different from our $p$) implies that
     \[\lim_{r\rightarrow 0^{+}} -r^{1/(\mu-1/2)}\mathbb{P}\Big(\sum_{l=-1}^{+\infty}(l^{-\mu}\xi_l)^2\leq cr^2\Big) = \tilde{C},\]
     for some constant $\tilde{C}>0$. Thus, there exist a constant $r_0>0$ such that
     for any $0<r\leq r_0$,
     \[
-\log\mathbb{P}\Big(\sum_{l=-1}^{+\infty}(l^{-\mu}\xi_l)^2\leq cr^2\Big) \leq 2\tilde{C}r^{-1/(\mu-1/2)} = 2\tilde{C}r^{-\frac{pd}{p(\alpha+\kappa)-d}}.\]
     For $r>r_0$, we have
     \[
-\log\mathbb{P}\Big(\sum_{l=-1}^{+\infty}(l^{-\mu}\xi_l)^2\leq cr^2\Big) \leq -\log\mathbb{P}\Big(\sum_{l=-1}^{+\infty}(l^{-\mu}\xi_l)^2\leq cr_0^2\Big) \leq 2\tilde{C}r_0^{-\frac{pd}{p(\alpha+\kappa)-d}}.\]
     In conclusion, let $C=2\tilde{C}(1\vee r_0^{-\frac{pd}{p(\alpha+\kappa)-d}})$, we deduce that
     \[-\log\Pi'(\norm{\theta}_{(H^{\kappa}(\mathbb{R}^d))^*}\leq r) \leq Cr^{-\frac{pd}{p(\alpha+\kappa)-d}},\]
     for any $0<r<1$.
\end{proof}
The next lemma is derived from the proofs of Lemma 6.4 in \cite{agapiou2024laplace} and Proposition 2.11 in \cite{agapiou2021rates}.
\begin{lemma}\label{lem:decenter}
    Consider $\Pi'$  defined in Definition \ref{Def:BesovPrior}. Then,for $p\in [1,2]$, any $h \in B_{pp}^{\alpha}(\mathbb{R}^d)$ supported on $K$ and any symmetric convex Borel-measurable $A\subset L_2(\mathbb{R}^d)$, it holds
    \[\Pi'(h+A)\geq e^{-\frac{1}{p}\norm{h}^p_{B_{pp}^{\alpha}}}\Pi'(A).\] 
\end{lemma}
\par
For any $p\geq 1$, Lemma \ref{lem:decenterJ} gives a similar result to Lemma \ref{lem:decenter} but restricted to the finite-dimensional prior $\Pi'_J$. 
\begin{lemma}\label{lem:decenterJ}
    Consider $\Pi'_J$ defined in Definition \ref{Def:TrunBesovPrior}. Then, for $p\geq 1$, any $h \in B_{pp}^{\alpha}(\mathbb{R}^d)$ supported on $K$ and any symmetric convex Borel-measurable $A\subset L_2(\mathbb{R}^d)$, there exists a constant $c>0$ such that
    \[\Pi'_J(P_J(h)+A)\geq \big(\frac{1}{2}\big)^{c2^{Jd}}e^{-\frac{1}{p}\norm{h}^p_{B_{pp}^{\alpha}}}\Pi'_J (A).\] 
\end{lemma}
\begin{proof}
    By Proposition 2.7 in \cite{agapiou2021rates} and definition of $\Pi'_J$, letting $V(x)=\frac{\abs{x}^p}{p}$, $\gamma_{lr} = 2^{-l(\alpha +\frac{d}{2}-\frac{d}{p})}$ and $f_{lr} = \pdt{f}{\psi_{lr}}_{L^2(\mathbb{R}^d)}$ for $f\in L^2(\mathbb{R}^d)$, we deduce that
    \begin{align*}
        &\Pi'_J(A+P_J(h)) = \int_A\mexp{\sum_{l=-1}^{J}\sum_{r\in R_l}\Big(V\big(\frac{\theta_{lr}}{\gamma_{lr}}\big)-V\big(\frac{\theta_{lr}-h_{lr}}{\gamma_{lr}}\big)\Big)}\Pi'_J(d\theta)\\
        \geq &e^{-\frac{1}{p}\norm{h}^p_{B_{pp}^{\alpha}}}\int_A\mexp{\sum_{l=-1}^{J}\sum_{r\in R_l}\Big(V\big(\frac{\theta_{lr}}{\gamma_{lr}}\big)+V\big(\frac{h_{lr}}{\gamma_{lr}}\big)-V\big(\frac{\theta_{lr}-h_{lr}}{\gamma_{lr}}\big)\Big)}\Pi'_J(d\theta)\\
        =&e^{-\frac{1}{p}\norm{h}^p_{B_{pp}^{\alpha}}}\int_A\exp\Big\{\sum_{l=-1}^{J-1}\sum_{r\in R_l}\Big(V\big(\frac{\theta_{lr}}{\gamma_{lr}}\big)+V\big(\frac{h_{lr}}{\gamma_{lr}}\big)-V\big(\frac{\theta_{lr}-h_{lr}}{\gamma_{lr}}\big)\Big)\\
        &\qquad\qquad\qquad+\sum_{r\in R_J/\{\tilde{r}\}}\Big(V\big(\frac{\theta_{Jr}}{\gamma_{Jr}}\big)+V\big(\frac{h_{Jr}}{\gamma_{Jr}}\big)-V\big(\frac{\theta_{Jr}-h_{Jr}}{\gamma_{Jr}}\big)\Big)\Big\}\\&\qquad\cdot\frac{1}{2}\Big(e^{V\big(\frac{\theta_{J\tilde{r}}}{\gamma_{J\tilde{r}}}\big)+V\big(\frac{h_{J\tilde{r}}}{\gamma_{J\tilde{r}}}\big)-V\big(\frac{\theta_{J\tilde{r}}-h_{J\tilde{r}}}{\gamma_{J\tilde{r}}}\big)}+e^{V\big(\frac{\theta_{J\tilde{r}}}{\gamma_{J\tilde{r}}}\big)+V\big(\frac{h_{J\tilde{r}}}{\gamma_{J\tilde{r}}}\big)-V\big(\frac{\theta_{J\tilde{r}}+h_{J\tilde{r}}}{\gamma_{J\tilde{r}}}\big)}\Big)\Pi'_J(d\theta),
    \end{align*}
where in the last line we used the symmetry of $V,A,\Pi'_J$.
Provided that, for $x,y\in R$,
\begin{equation}\label{lem:decenterJ:1}
    e^{V(x)+V(y)-V(x-y)}+e^{V(x)+V(y)-V(x+y)}\geq 1,
\end{equation}
we further deduce that
\begin{align*}
        &\Pi'_J(A+P_J(h)) \\\geq& e^{-\frac{1}{p}\norm{h}^p_{B_{pp}^{\alpha}}}\int_A\frac{1}{2}\exp\Big\{\sum_{l=-1}^{J-1}\sum_{r\in R_J}\Big(V\big(\frac{\theta_{lr}}{\gamma_{lr}}\big)+V\big(\frac{h_{lr}}{\gamma_{lr}}\big)-V\big(\frac{\theta_{lr}-h_{lr}}{\gamma_{lr}}\big)\Big)\\
        &\qquad\qquad\qquad+\sum_{r\in R_l/\{\tilde{r}\}}\Big(V\big(\frac{\theta_{Jr}}{\gamma_{Jr}}\big)+V\big(\frac{h_{Jr}}{\gamma_{Jr}}\big)-V\big(\frac{\theta_{Jr}-h_{Jr}}{\gamma_{Jr}}\big)\Big)\Big\}\Pi'_J(d\theta)\\
        \geq &e^{-\frac{1}{p}\norm{h}^p_{B_{pp}^{\alpha}}}\int_A\big(\frac{1}{2}\big)^{\abs{R_J}}\mexp{\sum_{l=-1}^{J-1}\sum_{r\in R_l}\Big(V\big(\frac{\theta_{lr}}{\gamma_{lr}}\big)+V\big(\frac{h_{lr}}{\gamma_{lr}}\big)-V\big(\frac{\theta_{lr}-h_{lr}}{\gamma_{lr}}\big)\Big)}\Pi'_J(d\theta)\\
        \geq &\big(\frac{1}{2}\big)^{\sum_{l=-1}^J\abs{R_l}}e^{-\frac{1}{p}\norm{h}^p_{B_{pp}^{\alpha}}}\Pi'_J(A)
        \geq \big(\frac{1}{2}\big)^{c2^{Jd}}e^{-\frac{1}{p}\norm{h}^p_{B_{pp}^{\alpha}}}\Pi'_J (A).
    \end{align*}
We complete our proof by proving \eqref{lem:decenterJ:1}.
Because $$e^{V(x)+V(y)-V(x-y)}+e^{V(x)+V(y)-V(x+y)}$$ is symmetric, it is sufficient to consider $x,y\geq 0$.
Notice that, when $x,y\geq 0$, 
$\abs{x-y}^p\leq x^p+y^p,$
which implies
$e^{V(x)+V(y)-V(x-y)} \geq 1.$
Thus, we have
\[e^{V(x)+V(y)-V(x-y)}+e^{V(x)+V(y)-V(x+y)}\geq 1.\]
\end{proof}
The next lemma is deduced directly from the proof of Lemma 6.5 in \cite{agapiou2024laplace} and the proof of Proposition 2.15 of \cite{agapiou2021rates}.
\begin{lemma}\label{lem:twolevel}
    Consider $\Pi'$ defined in Definition \ref{Def:BesovPrior}. Then there exists a constant $\Lambda>0$ such that for any $r>0$ 
    \begin{align*}
            &\Pi'\bigg(\theta = \theta_1 + \theta_2 +\theta_3: \theta_1\in A, \norm{\theta_2}_{B_{pp}^{\alpha}({\mathbb{R}^d})}\leq r^\frac{1}{p},\norm{\theta_3}_{H^{\alpha+\frac{d}{2}-\frac{d}{p}}({\mathbb{R}^d})} \leq \sqrt{r}, \\
            & \qquad \qquad \qquad \qquad \qquad \qquad  \theta_i\in \mathrm{span}\{\psi_{lr}\},i=1,2,3\bigg) \geq 1-\frac{1}{\Pi'(A)}\exp(-r/\Lambda).
    \end{align*}
\end{lemma}
We note that Lemmas \ref{lem:concentration}, \ref{lem:decenter}, \ref{lem:smallball}, and \ref{lem:twolevel} can also be proved for $\Pi_J'$ in place of $\Pi'$ with those constants independent of $J$. 

\section*{Acknowledgments}
This research was partially funded by the National Natural Science Foundation of China (Grant Nos. 12322116, 12271428, 12326606, and 42474139), the Fundamental and Interdisciplinary Disciplines Breakthrough Plan of the Ministry of Education of China (Grant No. JYB2025XDXM101), the National Key Research and Development Program of China (Grant No. 2022YFA1004100), and the Major Projects of the National Natural Science Foundation of China (Grant Nos. 12090021 and 12090020).

\bibliographystyle{plain}
\bibliography{reference}

\end{document}